\documentclass[a4paper,11pt, twoside]{article}
\usepackage[T1]{fontenc}
\usepackage[english]{babel}
\pdfoutput=1 

\usepackage{amsmath, amssymb, amsthm, amsfonts, mathtools, mathrsfs}	
\usepackage{cancel}
\usepackage{enumerate}
\usepackage[shortlabels]{enumitem}
\usepackage{geometry}
\usepackage{microtype}

\usepackage{xcolor}

\usepackage{cite} 

\usepackage{bbm}

\usepackage{fancyhdr} 

\newcommand{\R}{\mathbb{R}}
\newcommand{\N}{\mathbb{N}}

\numberwithin{equation}{section}

\theoremstyle{plain}
\newtheorem{theorem}{Theorem}[section] 
\newtheorem{proposition}[theorem]{Proposition} 

\theoremstyle{definition}
\newtheorem{definition}[theorem]{Definition}
\newtheorem{remark}[theorem]{Remark}

\DeclarePairedDelimiter{\abs}{\lvert}{\rvert}
\DeclarePairedDelimiter{\norm}{\lVert}{\rVert}
\DeclarePairedDelimiter{\duality}{\langle}{\rangle}

\DeclareMathOperator*{\essinf}{ess\,inf}

\DeclareMathOperator*{\argmin}{arg\,min}

\newcommand{\eps}{\varepsilon}
\renewcommand{\phi}{\varphi}
\renewcommand{\bar}{\overline}
\renewcommand{\vec}{\mathbf}

\def\genspazio #1#2#3#4#5{#1^{#2}(#5,#4;#3)}
\def\spazio #1#2#3{\genspazio {#1}{#2}{#3}T0}

\def\LT {\spazio L}
\def\HT {\spazio H}
\def\CT #1#2{C^{#1}([0,T];#2)}

\def\Lx #1{L^{#1}(\Omega)}
\def\Lt #1{L^{#1}(0,T)}
\def\Lqt #1{L^{#1}(Q_T)}
\def\Hx #1{H^{#1}(\Omega)}
\def\Wx #1{W^{#1}(\Omega)}
\def\Cx #1{C^{#1}(\bar{\Omega})}

\def\Lxod #1{L^{#1}(\Omega \setminus D)}
\def\Hxod #1{H^{#1}(\Omega \setminus D)}

\def\Accorpa #1#2 #3 {\gdef #1{\eqref{#2}--\eqref{#3}}%
	\wlog{}\wlog{\string #1 -> #2 - #3}\wlog{}}

\def\ls{<}
\def\gs{>}
\def\mezzo {\frac{1}{2}}
\def\ddt {\frac{\de}{\de t}}
\def\de {\mathrm{d}}
\def\D {\mathrm{D}}

\def\weakstar {\stackrel{\ast}{\rightharpoonup}}
\def\weak {\rightharpoonup}

\def\n {\vec{n}}
\def\phib {\bar{\phi}}
\def\mub {\bar{\mu}}
\def\ub {\bar{u}}
\def\Uad {\mathcal{U}_{\text{ad}}}
\def\Jcal {\mathcal{J}}
\def\lambdazb {\bar{\lambda}_0}
\def\lambdab {\bar{\lambda}}
\def\Scal {\mathcal{S}}
\def\Tcal {\mathcal{T}}
\def\XX {\mathbb{X}}
\def\YY {\mathbb{Y}}
\def\WW {\mathbb{W}}
\def\phih {\phi^h}
\def\muh {\mu^h}
\def\chiod {\chi_{\Omega \setminus D}}
\def\pb {\bar{p}}
\def\qb {\bar{q}}
\def\Jred {\mathcal{J}_{\text{red}}}
\def\hb {\bar{h}}
\def\ftil {\tilde{f}}
\def\phitil {\tilde{\phi}}
\def\mutil {\tilde{\mu}}

\usepackage{hyperref}

\begin{document}
	
	\begin{center}
		
		\LARGE{\textbf{Optimal control of the fidelity coefficient in a Cahn--Hilliard image inpainting model}}
		
		\vskip0.75cm
		
		\large{\textsc{Elena Beretta$^1$}} \\
		\normalsize{e-mail: \texttt{eb147@nyu.edu}} \\
		\vskip0.4cm
		
		\large{\textsc{Cecilia Cavaterra$^{2,3}$}} \\
		\normalsize{e-mail: \texttt{cecilia.cavaterra@unimi.it}} \\
		\vskip0.4cm
		
		\large{\textsc{Matteo Fornoni$^4$}} \\
		\normalsize{e-mail: \texttt{matteo.fornoni@unipv.it}} \\
		\vskip0.4cm
		
		\large{\textsc{Maurizio Grasselli$^5$}} \\
		\normalsize{e-mail: \texttt{maurizio.grasselli@polimi.it}} \\
		\vskip0.5cm
		
		\footnotesize{$^1$Division of Science, New York University Abu Dhabi, Saadiyat Island, Abu Dhabi, United Arab Emirates}
		\vskip0.3cm
		
		\footnotesize{$^2$Dipartimento di Matematica \lq\lq F. Enriques\rq\rq, Universit\`{a} degli Studi di Milano, 20133 Milano, Italy}
	\vskip0.3cm

        \footnotesize{$^3$Istituto di Matematica Applicata e Tecnologie Informatiche ``Enrico Magenes'', CNR, 27100 Pavia, Italy}
		\vskip0.3cm
        
		\footnotesize{$^4$Dipartimento di Matematica \lq\lq F. Casorati\rq\rq, Universit\`a degli Studi di Pavia, 27100 Pavia, Italy}
		\vskip0.3cm
		
		\footnotesize{$^5$Dipartimento di Matematica, Politecnico di Milano, 20133 Milano, Italy}
		\vskip0.75cm
		
	\end{center}

	\begin{abstract}\noindent
        We consider an inpainting model proposed by A.~Bertozzi et al., which is based on a Cahn--Hilliard-type equation. 
        This equation describes the evolution of an order parameter that represents an approximation of the original image occupying a bounded two-dimensional domain.
        The given image is assumed to be damaged in a fixed subdomain, and the equation is characterised by a linear reaction term. This term is multiplied by the so-called fidelity coefficient, which is a strictly positive bounded function defined in the undamaged region.
        The idea is that, given an initial image, the order parameter evolves towards the given image and this process properly diffuses through the boundary of the damaged region, restoring the damaged image, provided that the fidelity coefficient is large enough.
        Here, we formulate an optimal control problem based on this fact, namely, our cost functional accounts for the magnitude of the fidelity coefficient.
        Assuming a singular potential to ensure that the order parameter takes its values in between $0$ and $1$, we first analyse the control-to-state operator and prove the existence of at least one optimal control, establishing the validity of first-order optimality conditions. Then, under suitable assumptions, we demonstrate second-order optimality conditions.

		\vskip3mm
		
		
		\vskip3mm
		
		
	\end{abstract}

	\pagestyle{fancy}
	\fancyhf{}	
	\fancyhead[EL]{\thepage}
	\fancyhead[ER]{\textsc{Beretta -- Cavaterra -- Fornoni -- Grasselli}} 
	\fancyhead[OL]{\textsc{Optimal control for Cahn--Hilliard inpainting}} 
	\fancyhead[OR]{\thepage}

	\renewcommand{\headrulewidth}{0pt}
	\setlength{\headheight}{5mm}

	\thispagestyle{empty} 

	\section{Introduction}
	Image inpainting refers to the process of filling in missing or damaged regions of an image by leveraging information from the surrounding areas. Over the past few decades, many different techniques have been developed to address this goal. These include, for example, PDE-driven inpainting methods starting from the works of \cite{BertalmioSapiro2000} and \cite{shen2002mathematical};
    see also \cite{schonlieb2015partial} and references therein.
    A notable PDE model is the Cahn--Hilliard inpainting method introduced by Bertozzi et al. in \cite{BEG2006, BEG2007}. Here the evolution is governed by the Cahn--Hilliard equation coupled with a fidelity term that keeps the reconstructed image close to the given one in the undamaged regions. Other methods are based on variational principles where one minimizes an energy that balances fidelity to the known data and smoothness or sparsity constraints in the missing regions (see, e.g., \cite{chan2005image, guillemot2013image, scherzer2010handbook}). Other methods include exemplar-based inpainting techniques based on the idea that part of an image can be reconstructed by simply pasting samples extracted from the known part, and were introduced in \cite{criminisi2004region}. More recently deep-learning methods have become popular leveraging convolutional neural networks (CNNs) or generative adversarial networks (GANs) to predict missing content (see the review \cite{elharrouss2020image}). 
    In this paper, we follow the successful mathematical model introduced by Bertozzi et al. based on the Cahn--Hilliard equation, which ensures smooth transitions and sharp interfaces between known and missing regions. The inclusion of a fidelity term allows the model to stay close to the original image, while diffuse-interface modelling effectively preserves the edges. Indeed, it received a lot of attention in recent years and was generalized and studied from many points of view \cite{BDSW2014, BS2015, BHS2009, CFM2016, CFM2017, GLS2018, NR2022, THM2020}. Let us describe the mathematical model in detail.
    Given $\Omega \subset \R^2$ an open, bounded, and sufficiently regular domain (i.e. the image support), we interpret a binary image as a function $g: \Omega \to \{0,1\}$, where $g = 1$ stands for black and $g = 0$ for white. 
	We call $D \subset \Omega$ the damaged region and $g: \Omega \setminus D \to \{0,1\}$ the given original image.
	Our objective is to recover a phase-field approximation $u: \Omega \to [0,1]$ of the original image in the whole domain $\Omega$. 
	Notice that $u$ is allowed to also take the values in-between $0$ and $1$, but ideally only in a small diffuse interface of width $\eps$. 
	Then, we recover the inpainted image $\ub = u(T): \Omega \to [0,1]$ by solving the following initial and boundary value problem:
	\begin{alignat}{2}
		& \partial_t u - \Delta w = \lambda(x) (g - u) \qquad && \hbox{in $\Omega \times (0,T)$}, \label{eq:u_bertozzi} \\
		& w = - \eps \Delta u + \frac{1}{\eps} W'(u) && \hbox{in $\Omega \times (0,T)$}, \label{eq:w_bertozzi} \\
		& \partial_{\n} u = \partial_{\n} w = 0 && \hbox{on $\partial \Omega \times (0,T)$}, \label{bc_bertozzi} \\
		& u(0) = u_0 && \hbox{in $\Omega$}, \label{ic_bertozzi}
	\end{alignat}     
	where $\vec{n}$ is the outward unit normal vector to $\partial\Omega$. 
    Here $T \gs 0$ is a fixed final time, possibly large enough to ensure that $\ub$ is close enough to an equilibrium. Conversely, $u_0$ is a given initial data, to be thought of as a starting guess for the inpainted image (for instance, the convolution of $g$ with a suitable regularising kernel).
    
	The whole inpainting process essentially depends on the fidelity coefficient $\lambda$  whose typical form is
	\begin{equation}
		\label{eq:lambda_expr}
		\lambda(x) = \begin{cases}
			0 \quad & \hbox{if $x \in D$}, \\
			\lambda_0 \quad & \hbox{if $x \in \Omega \setminus D$},
		\end{cases}
	\end{equation}
	where $\lambda_0 \gs \gs 1$ constant. 
	Indeed, in $\Omega \setminus D$, the fidelity term $\lambda_0(g - u)$ forces the solution to stay close to the original image $g$, while, inside $D$, $u$ is just a solution to the standard Cahn--Hilliard equation. 
	We recall that Cahn--Hilliard equation can be seen as an $H^{-1}$ gradient flow of the Ginzburg--Landau energy 
	\[ E_\eps(u) = \int_\Omega \left( \frac{\eps}{2} \abs{\nabla u}^2 + \frac{1}{\eps} W(u) \right) \, \de x, \]  
	where $\eps \gs 0$ and $W$ is a double-well potential with equal minima at $0$ and $1$, which, for the moment, can be taken as follows $W(u) = u^2 (u - 1)^2$. 
	It is well-known (see, for instance, \cite{MM1977}) that, as $\eps \to 0$, the energy $E_\eps$ acts as a diffuse interface approximation, in the sense of $\Gamma$-convergence, of the perimeter energy, which was used in previous models for image inpainting (see, e.g., \cite{ES2002}).  
	Notice that we are not imposing any explicit boundary condition on $\partial D$, since the idea is that the sharp variation of $\lambda$ through $\partial D$ should already implicitly guide the restoration of the image inside of $D$ by matching the behaviour of $g$ on $\partial D$. 
	Indeed, in \cite{BEG2007}, it was shown that, if $D \subset\subset \Omega$, $g \in \mathcal{C}^2(\Omega)$ and $\lambda_0 \to +\infty$, the stationary solutions to \eqref{eq:u_bertozzi}--\eqref{ic_bertozzi} in $D$ converge to the solution of the following boundary value problem
	\begin{alignat*}{2}
		& - \Delta \left( - \eps \Delta u + \frac{1}{\eps} W'(u) \right) = 0 \qquad && \hbox{in $D$}, \\
		& u = g && \hbox{on $\partial D$}, \\
		& \nabla u = \nabla g && \hbox{on $\partial D$}, 
	\end{alignat*}
	where the last two equations are exactly the matching conditions needed for a good reconstruction.
    
    
    To further stress the importance of the choice of large values for $\lambda_0$, in the Appendix we complement the above-mentioned result by showing that also the solution to the evolutionary system \eqref{eq:u_bertozzi}--\eqref{ic_bertozzi} in $\Omega \setminus D$ 
    with boundary conditions on $\partial D$ given by the inpainted image, converges to a stationary solution in $\Omega\setminus D$ if $\lambda_0$ is large enough (cf. Proposition \ref{prop:asympt}).
	
	One possible drawback of model \eqref{eq:u_bertozzi}--\eqref{ic_bertozzi} is that, with a polynomial double-well potential as $W(u) = u^2 (u - 1)^2$, one cannot guarantee that the solution $u$ takes its values in $[0,1]$. 
	However, this issue can be solved by reformulating the problem using a physically relevant potential, namely, the so-called logarithmic-type double-well (or Flory--Huggins) potential (see \cite{CFM2015,L2022,M2021}).
    In some situations, potentials of this kind were also shown to provide better inpainting results and faster-converging algorithms (see \cite{CFM2015}).
	For convenience, we now change most of the notation used so far in order to work in the physical interval $[-1,1]$, instead of $[0,1]$, since it is more common in the literature on Cahn--Hilliard-type equations. 
    We call $\phi$ this new phase-field variable in $[-1,1]$.
	Clearly, this does not change anything in applications, since one can easily switch between the two formulations by setting
    $\varphi = 2(u-\frac{1}{2})$. More precisely, we replace $W(u)$ by
	\begin{equation}
		\label{eq:logpotential}
		F(\phi) = \frac{\theta}{2} \left( (1-\phi) \log(1-\phi) + (1 + \phi) \log(1 + \phi) \right) - \frac{\theta_c}{2} \phi^2 \quad \hbox{for $\phi \in [-1,1]$},
	\end{equation}
	which, for constants $0 \ls \theta \ls \theta_c$, has equal minima at the points $\pm m_*$, with $\abs{m_*} \ls 1$ characterised by the following implicit relation
	\[ \frac{1}{m_*} \log \left( \frac{1 - m_*}{1 + m_*} \right) = \frac{2 \theta_c}{\theta}. \] 
	Consequently, we now consider binary images of the form 
	\[  f: \Omega \setminus D \to \{ -m_*, m_* \}, \hbox{where $\{f = m_*\}$ is the black region and $\{f = - m_*\}$ the white one,} \]
	for a fixed constant $m_* \in (0,1)$. 
	Therefore, the inpainted image is now a function $\phib: \Omega \to [-1,1]$, which should always be approximately equal to $\pm m_*$ up to a diffuse interface of width $\eps$ (see \cite[Rem.~2.1]{L2022}).
	Indeed, we find $\phib = \phi(T)$ by solving the following evolutionary problem:   
	\begin{alignat}{2}
		& \partial_t \phi - \Delta \mu = \lambda(x) (f - \phi) \qquad && \hbox{in $\Omega \times (0,T)$}, \label{eq:phi0} \\
		& \mu = - \eps \Delta \phi + \frac{1}{\eps} F'(\phi) && \hbox{in $\Omega \times (0,T)$}, \label{eq:mu0} \\
		& \partial_{\n} \phi = \partial_{\n} \mu = 0 && \hbox{on $\partial \Omega \times (0,T)$}, \label{eq:bc0} \\
		& \phi(0) = \phi_0 && \hbox{in $\Omega$}, \label{eq:ic0}
	\end{alignat}     
	for given $T \gs 0$ large enough and $\phi_0: \Omega \to [-1,1]$. 
	The main advantage of this reformulated problem is that, due to the blow-up of $F'(\phi)$ for $\phi \to \pm 1$, any energy bounded weak solution, which necessarily requires the integrability of $F'(\phi)$, is such that $\abs{\phi} \ls 1$ for almost any $(x,t) \in \Omega \times (0,T)$. 
	
	The goal of this work is to propose and study an optimal control problem aimed at finding a fidelity term that guarantees a faithful inpainting of the given damaged image $f$.
    On account of \eqref{eq:lambda_expr}, we seek more generally a function $\lambda$ of the form 
	\begin{equation}
		\label{eq:lambda_gen}
		\lambda(x) = \begin{cases}
			0 \quad & \hbox{if $x \in D$}, \\
			\lambda_0(x) \quad & \hbox{if $x \in \Omega \setminus D$},
		\end{cases}
	\end{equation}
	where $\lambda_0 \in L^\infty(\Omega \setminus D)$ is such that $\essinf \lambda_0 \ge \lambda_{\text{min}} \gs 0$. 
	We allow $\lambda_0$ to be dependent on $x \in \Omega$, as it could be beneficial to have different values for it. For instance, it might be useful to take $\lambda_0$ larger in regions where the image $f$ carries a higher level of detail, and it might be taken smaller in less detailed regions. Our optimal control problem can take care of this aspect. 
	Indeed, we fix $0 \ls \lambda_{\text{min}} \ls \lambda_{\text{max}}$ as a minimum and a maximum value for $\lambda_0$ and we formulate the following problem:
	
	\bigskip
	\noindent(CP) \textit{Minimise the cost functional}
	\begin{equation}
			\mathcal{J}(\phi, \lambda_0) = \frac{\alpha_1}{2} \int_0^T \int_{\Omega \setminus D} |\phi - f|^2 \,\de x \, \de t + \frac{\alpha_2}{2} \int_{\Omega \setminus D} |\phi(T) - f|^2 \,\de x + \frac{\beta}{2} \int_{\Omega \setminus D} \, \left(\frac{1}{\lambda_0}\right)^2 \,\de x,
	\end{equation}
	\textit{subject to the control constraints}
	\begin{equation} 
			\lambda_0 \in \Uad := \{ \lambda_0 \in L^{\infty}(\Omega \setminus D) \mid 0 < \lambda_{\text{min}} \le \lambda_0 \le \lambda_{\text{max}} \, \text{ a.e. in } \Omega \},  
	\end{equation}
	\textit{and to the state system \eqref{eq:phi0}-\eqref{eq:ic0}}.
	\bigskip
	
	\noindent 
    Here, $\alpha_1$, $\alpha_2$ and $\beta$ are non-negative parameters that can be tuned depending on the particular application.
    The first two tracking-type terms in the cost functional are there to keep the solution close to the original image $f$ in the undamaged domain, both during its evolution and at the final time.
    Conversely, note that the third term in $\mathcal{J}$ actually promotes large values of $\lambda_0$, as it is expected from the analysis carried out in \cite{BEG2007}.
	Also, observe that it does not hinder the convexity of the functional $\mathcal{J}$ on the admissible set $\Uad$.
    We conclude by recalling that a similar optimal control problem for the related Cahn--Hilliard--Oono equation
    is studied in \cite{CGRS2022}. There, the control is a space-time mass density and first-order optimality conditions are obtained.
	
    Let us describe our main results. Preliminarily, we report some results on the well-posedness of problem \eqref{eq:phi0}-\eqref{eq:ic0} which allow us to define a suitable control-to-state operator (see Section~\ref{WP}).
    Then, we show the existence of an optimal control in Section~\ref{EOC} and we establish first-order optimality conditions in Section~\ref{OC1}. Finally, we prove second-order optimality conditions in Section~\ref{OC2}.
    


	\section{Preliminaries and well-posedness results}
    \label{WP}

	We begin by introducing some notation that will be used throughout the paper. 
	We first recall that $\Omega \subset \R^2$ is an open and bounded domain with sufficiently regular boundary.
	Let $T > 0$ be the final time and denote $Q_t = \Omega \times (0, t)$ and $\Sigma_t = \partial \Omega \times (0, t)$, for any $t \in (0, T]$.
	For a (real) Banach space $X$, we denote by $\norm{\cdot}_X$ its norm, by $X^*$ its dual space and by $\duality{\cdot, \cdot}_X$ the duality pairing between $X$ and $X^*$. 
	If $X$ is a Hilbert space, we additionally denote by $(\cdot, \cdot)_X$ its inner product.
	For the standard Lebesgue and Sobolev spaces with $1 \le p \le \infty$ and $k \in \N$, we use the notation $\Lx p$ and $\Wx {k,p}$, with norms $\norm{\cdot}_{\Lx p}$ and $\norm{\cdot}_{\Wx {k,p}}$ respectively. 
	In the case $p=2$, we employ the notation $\Hx k := \Wx {k, 2}$ and denote its norm by $\norm{\cdot}_{\Hx k}$.
	For Bochner spaces, we use the standard notation $\LT p X$, for any $1 \le p \le \infty$ and any Banach space $X$.
	
	Next, we introduce the main Hilbert triplet that will be used in the following. 
	Namely, we set
	\[
		H := \Lx 2, \quad V := \Hx 1, \quad W := \{ u \in \Hx 2 \mid \partial_{\n} u = 0 \},
	\]
	and we recall that the following compact and dense embeddings hold:
	\[
		W \hookrightarrow V \hookrightarrow H \cong H^* \hookrightarrow V^* \hookrightarrow W^*.
	\]
	Moreover, owing to standard elliptic regularity estimates, we can equip $W$ with the equivalent norm
	\[ 
		\norm{w}^2_W := \norm{w}^2_H + \norm{\Delta w}^2_H, 
	\]
	where $\Delta$ denotes the Laplace operator.
	We also need to define the generalised mean value of a function $v$ as 
	\[
		v_\Omega := \frac{1}{\abs{\Omega}} \int_\Omega v \, \de x \hbox{ if $v \in \Lx 1$}
		\quad \hbox{and} \quad
		v_\Omega := \frac{1}{\abs{\Omega}} \duality{v, 1}_V \hbox{ if $v \in V^*$}, 
	\]
	$\vert \Omega\vert$ being the Lebesgue measure of $\Omega$, and introduce the zero-mean spaces
	\[
		V_0 := \{ v \in V \mid v_\Omega = 0 \}, \quad 
		V^*_0 := \{ v^* \in V^* \mid v^*_\Omega = 0 \}.
	\]
	Then, the Neumann--Laplace operator $\mathcal{N}: V_0 \to V^*_0$, defined as 
	\[
		\duality{\mathcal{N}w, v} := \int_\Omega \nabla w \cdot \nabla v \, \de x,
	\]
	is positive definite and self-adjoint. 
	In particular, by the Lax--Milgram theorem and Poincar\'e--Wirtinger's inequality, the inverse operator $\mathcal{N}^{-1}: V^*_0 \to V_0$ is well defined and $w = \mathcal{N}^{-1} g$ if $w \in V_0$ is the weak solution to the elliptic boundary value problem
	\[
		\begin{cases}
			- \Delta w = g \quad & \hbox{in $\Omega$,} \\
			\partial_{\n} w = 0 \quad & \hbox{on $\partial \Omega$.}
		\end{cases}
	\]
	Finally, we recall that by spectral theory the eigenfunctions of the operator $\mathcal{N}$ form an orthonormal Schauder basis in $H$, which is also orthogonal in $V$. 


\subsection{Weak and strong solutions}

    We introduce the following structural assumptions.
	\begin{enumerate}[font = \bfseries, label = A\arabic*., ref = \bf{A\arabic*}]
		\item\label{ass:omega} $\Omega \subset \R^2$ is an open bounded domain with boundary $\partial \Omega$ of class $\mathcal{C}^4$ and $D \subset \Omega$ is a bounded domain such that $0< \vert D\vert < \vert \Omega\vert$.
		\item\label{ass:F} The potential $F$ admits the splitting $F = F_0 + F_1$. We assume that $F_0: [-1,1] \to [0, +\infty)$ is a continuous strictly convex function such that
		\begin{equation}
			\label{hp:F0_general}
			\begin{split}
				& \quad F_0 \in \mathcal{C}^4(-1,1), 
				\quad F_0(0) = F'_0(0) = 0, \\
				& \qquad F_0(-s) = F_0(s) \quad \forall s \in (-1,1), \\
				& \lim_{s \to \pm 1} F_0'(s) = \pm \infty, 
				\quad \lim_{s \to \pm 1} F_0''(s) = + \infty.
			\end{split}
		\end{equation}
		We also suppose that there exist constants $c_0, C > 0$ such that for any $s \in (-1,1)$ the following bounds hold:
		\begin{equation}
			\label{hp:F0_bounds}
			  F_0''(s) \ge c_0 > 0, 
				\quad F_0''(s) \le e^{C (1 + \abs{F_0'(s)})}. 
		\end{equation}
		Instead, $F_1: [-1,1] \to \R$ is a $W^{4,\infty}$-perturbation such that $\abs{F_1''(s)} \le \theta$ for any $s \in [-1,1]$, for some $\theta > 0$.
		\item\label{ass:f} $f \in L^\infty(\Omega \setminus D)$ and there exists a positive constant $\eps_0 > 0$ such that $\norm{f}_{L^\infty(\Omega \setminus D)} \le 1 - \eps_0$. 
		For the purposes of our analysis, we can also assume that $f$ is extended by $0$ to a function in $\Lx\infty$, which we still call $f$ for simplicity. 
		\item\label{ass:lambda} We assume that the fidelity coefficient $\lambda \in \Lx\infty$ has the form given in \eqref{eq:lambda_gen}, i.e.
		\begin{equation*}
			\lambda(x) = \begin{cases}
				0 \quad & \hbox{if $x \in D$}, \\
				\lambda_0(x) \quad & \hbox{if $x \in \Omega \setminus D$},
			\end{cases}
		\end{equation*}
		and that there exists $\lambda_{\text{min}} > 0$ such that $\lambda(x) \ge \lambda_{\text{min}}$ in $D$.
		\item\label{ass:iniz} $\phi_0 \in V$ satisfies $\norm{\phi_0}_{\Lx\infty} \le 1$, $F_0(\phi_0) \in \Lx 1$ and $\abs{(\phi_0)_\Omega} < 1$. 
	\end{enumerate}
	In the sequel, $C > 0$ will stand for a generic positive constant, possibly changing from line to line, depending only on the fixed parameters introduced above.

	\begin{remark}
		Assumptions \ref{ass:omega}--\ref{ass:iniz} are the same used in \cite{L2022} to show the existence of strong solutions to \eqref{eq:phi0}--\eqref{eq:ic0}, as we will recall below. 
		Observe that \ref{ass:F} is satisfied by the logarithmic double-well potential \eqref{eq:logpotential} introduced above. 
		We also mention that \eqref{hp:F0_bounds} can be further weakened in view of the recent results regarding the validity of the separation property in two spatial dimensions (see \cite[Thm.~3.3]{GP2024}). Moreover, for \ref{ass:f}, we refer again to \cite[Rem.~2.1]{L2022}.
	\end{remark}

	We first recall the result contained in \cite[Proposition 2.5]{L2022} on the existence of weak solutions to \eqref{eq:phi0}--\eqref{eq:ic0} (cf. also \cite{M2021}).

	\begin{proposition}
		\label{prop:weaksols}
		Let \ref{ass:omega}--\ref{ass:iniz} hold.
		Then, there exists a pair $(\phi, \mu)$ to \eqref{eq:phi0}--\eqref{eq:ic0} enjoying the regularity
		\begin{align*}
			& \phi \in \HT 1 {V^*} \cap \LT \infty V \cap \LT 4 W \cap \LT 2 {\Wx{2,p}} \quad \hbox{for any $p \in [2,+\infty)$,} \\
			& \mu \in \LT 2 V, \quad F_0'(\phi) \in \LT 2 H, \\
			& \abs{\phi(x,t)} < 1 \quad \hbox{for a.e.~$(x,t) \in Q_T$,} \quad \abs{\phi_\Omega(t)} < 1 \quad \hbox{for any $t \in (0,T)$,}
		\end{align*}
		and such that 
		\begin{align*}
			& \duality{\partial_t \phi, \zeta}_V + (\nabla \mu, \nabla \phi)_H = (\lambda(x)(f - \phi), \zeta)_H \quad \hbox{for any $\zeta \in V$, a.e. in $(0,T)$},  \\
			& \mu = - \eps \Delta \phi + \frac{1}{\eps} F'(\phi) \quad \hbox{a.e.~in $Q_T$,}\\
            & \phi(0) = \phi_0 \quad \hbox{ in } V. 
		\end{align*}
		Such a pair is called (finite energy) weak solution to problem \eqref{eq:phi0}--\eqref{eq:ic0}.
	\end{proposition}

	Next, we recall the result we will use in the sequel to define the control-to-state operator, namely, the well-posedness of problem  \eqref{eq:phi0}--\eqref{eq:ic0} with respect to strong solutions (cf. \cite[Theorem 3.2]{L2022}). We point out that a key step to obtain higher regularity for solutions to \eqref{eq:phi0}--\eqref{eq:ic0} is to establish the \emph{strict separation property}, that is, proving that $\phi$ stays uniformly away from $\pm 1$ in $[0,T]$ (see \cite{GP2024} and references therein). This basic property allows us to handle the singular potential as if it were a globally Lipschitz nonlinearity.

	\begin{theorem}
		\label{thm:strongsols}
		Let \ref{ass:omega}--\ref{ass:iniz} hold and assume that $\phi_0\in W$ and  $\mu_0 := - \Delta \phi_0 + F'(\phi_0)$ satisfies $\nabla \mu_0 \in H$. 
		Then, a weak solution $(\phi, \mu)$ given by Proposition \emph{\ref{prop:weaksols}} satisfies the further regularities 
		\begin{align*}
			& \phi \in \HT 1 V \cap \LT \infty {\Wx{2,p}} \cap \LT 2 {\Hx4} \quad \hbox{for any $p \ge 2$,} \\
			& \mu \in \LT \infty V \cap \LT 2 W, \\
			& F'(\phi), F''(\phi) \in \LT \infty {\Lx p} \quad \hbox{for any $p \ge 2$.} 
		\end{align*}
        Moreover, the strict separation property holds, namely there exists $\delta > 0$, depending only on the data of the system, such that 
        \begin{equation}
			\label{eq:separation}
			\norm{\phi}_{C^0(\bar{Q_T})} \le 1 - \delta,
		\end{equation}
        Finally, such a solution is \emph{unique}.
	\end{theorem}

	\begin{remark}
		\label{rmk:regularity}
        In the original \cite[Theorem 3.2]{L2022} only the \emph{instantaneous} strict separation property was shown, namely that for any $\sigma > 0$ there exists $\delta > 0$ such that 
		\begin{align*}
			\abs{\phi(x,t)} \le 1 - \delta \quad \hbox{for any $(x,t) \in \bar{\Omega} \times [\sigma, T]$.} 
		\end{align*}
		However, we observe that our assumptions on $\phi_0$ and $\mu_0$ already entail the existence of a $0 < \delta_0 < 1$ such that $\phi$ is strictly separated locally in time (see \cite[Rem.~3.9]{GP2022} and references therein). Thus the strict separation property holds up to time $t = 0$ (see \cite[Rem.~3.11]{GP2022}), as shown in \eqref{eq:separation} for some $0 < \delta < \delta_0$. 
        In particular, we stress that $\delta_0$ depends only on the initial energy, therefore it can be chosen uniformly if the energy of the initial guess $\phi_0$ is uniformly bounded.
    
	\end{remark}

	The analysis conducted in \cite{L2022} shows, in particular, that the solutions to \eqref{eq:phi0}--\eqref{eq:ic0} regularise in finite time (see \cite[Thm.~3.1]{L2022}). 
	Hence, for the purposes of our analysis, we can assume to start with an initial guess $\phi_0$ satisfying 
	\begin{enumerate}[font = \bfseries, label = A\arabic*., ref = \bf{A\arabic*}]
		\setcounter{enumi}{5}
		\item\label{ass:initial2} $\phi_0 \in W$ and $\mu_0\in V$.
	\end{enumerate}	
	In this way, we can use all the global regularity properties shown in Theorem \ref{thm:strongsols} and the separation property \eqref{eq:separation}, namely, we have a unique strong solution. 
	We stress that \ref{ass:initial2} is not restrictive for the inpainting algorithm since, even if the initial guess $\phi_0$ is less regular, we can take as a ``new'' initial guess the more regular $\tilde{\phi}_0 = \phi(\sigma)$, where $\sigma$ could be any positive time as in Remark \ref{rmk:regularity}. 

	\begin{remark}
		\label{rmk:Fderivatives}
		Observe that, thanks to \ref{ass:initial2}  and Remark \ref{rmk:regularity}, there exists a constant $C > 0$ such that 
		\begin{equation}
        \label{derbounds}
			\norm{F^{(i)}(\phi)}_{C^0(\bar{Q_T})} \le C \quad \hbox{for any $i = 0, \dots, 4$.}
		\end{equation}
	\end{remark}

    \subsection{Continuous dependence on the fidelity coefficient}

    For the subsequent analysis, we fix $\eps = 1$ in \eqref{eq:phi0}--\eqref{eq:ic0} as its particular value does not influence our results. 
	Hence, from now on we consider the following initial and boundary value problem:
	\begin{alignat}{2}
		& \partial_t \phi - \Delta \mu = \lambda(x) (f - \phi) \qquad && \hbox{in $Q_T$}, \label{eq:phi} \\
		& \mu = - \Delta \phi + F'(\phi) && \hbox{in $Q_T$}, \label{eq:mu} \\
		& \partial_{\n} \phi = \partial_{\n} \mu = 0 && \hbox{on $\Sigma_T$}, \label{eq:bc} \\
		& \phi(0) = \phi_0 && \hbox{in $\Omega$}. \label{eq:ic}
	\end{alignat}
	In order to study the optimal control problem (CP) in the next section, we need to define a control-to-state operator. Along with the existence and uniqueness of a strong solution
    given by Theorem \ref{thm:strongsols} and the validity of the separation property \eqref{eq:separation}, we also need to establish 
	a continuous dependence of the (strong) solution itself on the fidelity coefficient $\lambda$.  

	\begin{theorem}
		\label{thm:contdep}
		Assume that \ref{ass:omega}--\ref{ass:f} and \ref{ass:initial2} hold. Let $\lambda_1$, $\lambda_2$ be two fidelity coefficients satisfying \ref{ass:lambda} and let $(\phi_1, \mu_1)$, $(\phi_2, \mu_2)$ be the corresponding (strong) solutions to \eqref{eq:phi}--\eqref{eq:ic}. Then, there exists a constant $K > 0$, depending only on the parameters of the system and on $\lambda_1$, $\lambda_2$, but not on their difference, such that the following continuous dependence estimate holds:
		\begin{equation}
			\label{eq:contdep:result} 
			\norm{\phi_1 - \phi_2}^2_{\HT 1 H \cap \LT \infty W \cap \LT 2 {\Hx4}} + \norm{\mu_1 - \mu_2}^2_{\LT 2 W} 
			\le K \norm{\lambda_1 - \lambda_2}^2_H.
		\end{equation}
	\end{theorem}

	\begin{proof}
        {\allowdisplaybreaks
		We argue as in the proof of \cite[Theorem 3.2]{L2022}. Let us set
		  
		\[
			\phi := \phi_1 - \phi_2, \quad \mu := \mu_1 - \mu_2, \quad \lambda := \lambda_1 - \lambda_2.
		\]
		Then, it is easy to observe that
		\begin{alignat}{2}
			& \partial_t \phi - \Delta \mu = \lambda (f - \phi_2) - \lambda_1 \phi \quad && \hbox{in $Q_T$,} \label{eq:phi2} \\
			& \mu = - \Delta \phi + (F'(\phi_1) - F'(\phi_2)) \quad && \hbox{in $Q_T$,} \label{eq:mu2} \\
			& \partial_{\n} \phi = \partial_{\n} \mu = 0 \quad && \hbox{on $\Sigma_T$,} \label{eq:bc2} \\
			& \phi(0) = 0 \quad && \hbox{in $\Omega$.} \label{eq:ic2}
		\end{alignat}
		\textsc{First estimate.}
		We multiply equation \eqref{eq:phi2} by $\mu$  and equation \eqref{eq:mu2} by $\partial_t \phi - \Delta \phi$, we integrate over $\Omega$ and we add the obtained equations together. This gives
		\begin{equation}
			\label{eq:contdep:est1}
			\begin{split}
			& \mezzo \ddt \norm{\nabla \phi}^2_H + \norm{\nabla \mu}^2_H + \norm{\Delta \phi}^2_H + (F'(\phi_1) - F'(\phi_2), \partial_t \phi)_H \\
			& \quad = (\lambda (f - \phi_2), \mu)_H - (\lambda_1 \phi, \mu)_H - (\mu, \Delta \phi)_H + (F'(\phi_1) - F'(\phi_2), \Delta \phi)_H. 
			\end{split}
		\end{equation}
		We start by estimating the term on the left-hand side. 
		Using hypothesis \ref{ass:F} and Leibniz's rule on the time-derivative, we infer that 
		\begin{align*}
			& (F'(\phi_1) - F'(\phi_2), \partial_t \phi)_H \\
			& \quad = (F_0'(\phi_1) - F_0'(\phi_2), \partial_t \phi)_H
			+ (F_1'(\phi_1) - F_1'(\phi_2), \partial_t \phi)_H \\
			& \quad = \ddt (F_0'(\phi_1) - F_0'(\phi_2),  \phi)_H 
			- (F_0''(\phi_1) \partial_t \phi_1 - F_0''(\phi_2) \partial_t \phi_2, \phi)_H 
			+ (F_1'(\phi_1) - F_1'(\phi_2), \partial_t \phi)_H \\
			& \quad = \ddt (F_0'(\phi_1) - F_0'(\phi_2),  \phi)_H 
			- (F_0''(\phi_1) \partial_t \phi, \phi)_H \\
			& \qquad - ( (F_0'(\phi_1 ) - F_0'(\phi_2)) \partial_t \phi_2, \phi )_H 
			+ (F_1'(\phi_1) - F_1'(\phi_2), \partial_t \phi)_H \\
			& \quad = \ddt (F_0'(\phi_1) - F_0'(\phi_2),  \phi)_H - I_1 - I_2 + I_3.
		\end{align*}
		To estimate $I_1$, we first rewrite $\partial_t \phi$ by exploiting equation \eqref{eq:phi2}.
		Then, we use integration by parts, H\"older's inequality, Young's inequality and Sobolev embeddings, together with the regularity properties guaranteed by Theorem \ref{thm:strongsols}, \eqref{eq:separation} and \eqref{derbounds}, to deduce that
		\begin{align*}
			I_1 & = \int_\Omega F_0''(\phi_1) \partial_t \phi \, \phi \, \de x \\
			& = \int_\Omega F_0''(\phi_1) \Delta \mu \, \phi \, \de x
			+ \int_\Omega F_0''(\phi_1) \lambda f \phi \, \de x 
			- \int_\Omega F_0''(\phi_1) \lambda_1 \phi^2 \, \de x 
			- \int_\Omega F_0''(\phi_1) \lambda \phi_2 \, \phi \, \de x \\
			& = - \int_\Omega F_0'''(\phi_1) \nabla \phi_1 \cdot \nabla \mu \, \phi \, \de x 
			- \int_\Omega F_0''(\phi_1) \nabla \phi \cdot \nabla \mu \, \de x \\
			& \quad + \int_\Omega F_0''(\phi_1) \lambda f \phi \, \de x 
			- \int_\Omega F_0''(\phi_1) \lambda_1 \phi^2 \, \de x 
			- \int_\Omega F_0''(\phi_1) \lambda \phi_2 \, \phi \, \de x \\
			& \le \norm{F_0'''(\phi_1)}_{\Lx\infty} \norm{\nabla \phi_1}_{\Lx 4} \norm{\nabla \mu}_H \norm{\phi}_{\Lx 4}
			+ \norm{F_0''(\phi_1)}_{\Lx\infty} \norm{\nabla \phi}_H \norm{\nabla \mu}_H \\
			& \quad + \norm{F_0''(\phi_1)}_{\Lx\infty} \left( \norm{\lambda}_H \norm{f}_{\Lx\infty} \norm{\phi}_H 
			+ \norm{\lambda_1}_{\Lx\infty} \norm{\phi}^2_H
			+ \norm{\lambda}_H \norm{\phi_2}_{\Lx\infty} \norm{\phi}_H \right) \\
			& \le C \norm{\phi_1}_{\Hx 2} \norm{\nabla \mu}_H \norm{\phi}_V 
			+ C \norm{\nabla \phi}_H \norm{\nabla \mu}_H \\
			& \quad + C \norm{\lambda}_H \norm{\phi}_H 
			+ C \norm{\phi}^2_H 
			+ C \norm{\lambda}_H \norm{\phi}_H \\
			& \le \frac14 \norm{\nabla \mu}^2_H
			+ C \left( 1 + \norm{\phi_1}^2_{\Hx2} \right) \norm{\phi}^2_V
			+ C \norm{\lambda}^2_H. 
		\end{align*}
		Here we used the Sobolev embeddings $V \hookrightarrow \Lx4$ and $\Hx2 \hookrightarrow \Wx{1,4}$. 
		Moreover, by Theorem \ref{thm:strongsols}, we also know that $\norm{\phi_1}^2_{\Hx2} \in \Lt\infty$ uniformly.
		In order to estimate $I_2$, we use H\"older's inequality, Young's inequality and Sobolev embeddings, together with the separation property \eqref{eq:separation}, the local Lipschitz continuity of $F_0'$ and Remark \ref{rmk:Fderivatives}. 
		Hence, we obtain that 
		\begin{align*}
			I_2 & = \int_\Omega (F_0'(\phi_1) - F_0'(\phi_2)) \partial_t \phi_2 \, \phi \, \de x 
			\le \norm{F_0'(\phi_1) - F_0'(\phi_2)}_{\Lx4} \norm{\partial_t \phi_2}_H \norm{\phi}_{\Lx4} \\
			& \le C \norm{\phi}^2_{\Lx4} \norm{\partial_t \phi_2}_H 
			\le C \norm{\partial_t \phi_2}_H \norm{\phi}^2_V,  
		\end{align*}
		where $\norm{\partial_t \phi_2}_H \in \Lt2$ thanks to Theorem \ref{thm:strongsols}.
		To estimate $I_3$, we use again \eqref{eq:phi2} to replace $\partial_t \phi$. Thus we get
		\begin{align*}
			I_3 & = \int_\Omega (F_1'(\phi_1) - F_1'(\phi_2)) \partial_t \phi \, \de x \\
			& = \int_\Omega (F_1'(\phi_1) - F_1'(\phi_2)) (\Delta \mu + \lambda f - \lambda_1 \phi - \lambda \phi_2) \, \de x \\
			& = - \int_\Omega F_1''(\phi_1) \nabla \phi \cdot \nabla \mu \, \de x
			- \int_\Omega (F_1''(\phi_1) - F_1''(\phi_2)) \nabla \phi_2 \cdot \nabla \mu \, \de x \\
			& \quad + \int_\Omega (F_1'(\phi_1) - F_1'(\phi_2)) (\lambda f - \lambda_1 \phi - \lambda \phi_2) \, \de x \\
			& \le \norm{F''_1(\phi_1)}_{\Lx\infty} \norm{\nabla \phi}_H \norm{\nabla \mu}_H 
			+ \norm{F_1''(\phi_1) - F_1''(\phi_2)}_{\Lx4} \norm{\nabla \phi_2}_{\Lx4} \norm{\nabla \mu}_{\Lx4} \\
			& \qquad + \norm{F_1''(\phi_1) - F_1''(\phi_2)}_{H} \left( \norm{\lambda}_H \norm{f}_{\Lx\infty} + \norm{\lambda_1}_{\Lx\infty} \norm{\phi}_H + \norm{\lambda}_H \norm{\phi_2}_{\Lx\infty} \right) \\
			& \le C \norm{\nabla \phi}_H \norm{\nabla \mu}_H 
			+ C \norm{\phi}_V \norm{\phi_2}_{\Hx2} \norm{\nabla \mu}_{H} + C \norm{\phi}_H \left( \norm{\lambda}_H + \norm{\phi}_H \right) \\
			& \le \frac14 \norm{\nabla \mu}^2_H + C \left(1 + \norm{\phi_2}^2_{\Hx2} \right) \norm{\phi}^2_V + C \norm{\lambda}^2_H,
		\end{align*}
		where $\norm{\phi_2}^2_{\Hx2} \in \Lt\infty$ owing to Theorem \ref{thm:strongsols}.
		We now estimate all the remaining terms on the right-hand side of \eqref{eq:contdep:est1}. 
		Here we use a combination of Cauchy--Schwarz and Young's inequalities and as well as the separation property \eqref{eq:separation}, the local Lipschitz continuity of $F'$, and \eqref{derbounds}. 
		We will also replace $\mu$ by exploiting equation \eqref{eq:mu2} for the first three terms.
		Summing up, we have   
		\begin{align*}
			& (\lambda f, \mu)_H = (\lambda f, - \Delta \phi + (F'(\phi_1) - F'(\phi_2)))_H \\
			& \qquad \le \norm{\lambda}_H \norm{f}_{\Lx\infty} \norm{\Delta \phi}_H + \norm{\lambda}_H \norm{f}_{\Lx\infty} \norm{F'(\phi_1) - F'(\phi_2)}_H \\
			& \qquad \le \frac18 \norm{\Delta \phi}^2_H + C \norm{\phi}^2_H + C \norm{\lambda}^2_H, \\
			& (\lambda \phi_2, \mu)_H = (\lambda f, - \Delta \phi + (F'(\phi_1) - F'(\phi_2)))_H \\
			& \qquad \le \norm{\lambda}_H \norm{\phi_2}_{\Lx\infty} \norm{\Delta \phi}_H + \norm{\lambda}_H \norm{\phi_2}_{\Lx\infty} \norm{F'(\phi_1) - F'(\phi_2)}_H \\
			& \qquad \le \frac18 \norm{\Delta \phi}^2_H + C \norm{\phi}^2_H + C \norm{\lambda}^2_H, \\  
			& (\lambda_1 \phi, \mu)_H  = (\lambda_1 \phi, - \Delta \phi + (F'(\phi_1) - F'(\phi_2)))_H \\
			& \qquad \le \norm{\lambda_1}_{\Lx\infty} \norm{\phi}_H \norm{\Delta \phi}_H + \norm{\lambda_1}_{\Lx\infty} \norm{\phi}_H \norm{F'(\phi_1) - F'(\phi_2)}_H \\
			& \qquad \le \frac18 \norm{\Delta \phi}^2_H + C \norm{\phi}^2_H, \\
			& - (\mu, \Delta \phi)_H = (\nabla \mu, \nabla \phi)_H 
			\le \frac14 \norm{\nabla \mu}^2_H + C \norm{\nabla \phi}^2_H \\
			& (F'(\phi_1) - F'(\phi_2), \Delta \phi)_H \le \frac18 \norm{\Delta \phi}_H + C \norm{\phi}^2_H.
		\end{align*}  
		Consequently, by integrating \eqref{eq:contdep:est1} on $(0,t)$, for any $t \in (0,T)$, using all the above estimates and \eqref{eq:ic2}, we infer that 
		\begin{equation}
			\label{eq:contdep:est1int}
			\begin{split}
			& \mezzo \norm{\nabla \phi(t)}^2_H 
			+ \int_\Omega (F'_0(\phi_1) - F'_0(\phi_2)) \phi(t) \, \de x 
			+ \frac14 \int_0^t \norm{\nabla \mu}^2_H \, \de s 
			+ \mezzo \int_0^t \norm{\Delta \phi}^2_H \, \de s \\
			& \quad \le C \int_0^T \left( 1 + \norm{\phi_1}^2_{\Hx2} + \norm{\phi_2}^2_{\Hx2} + \norm{\partial_t \phi_2}_H \right) \norm{\phi}^2_V \, \de s + CT \norm{\lambda}^2_H.
			\end{split}
		\end{equation}
		At this point, by using the strict convexity of $F_0$ (see   \eqref{hp:F0_bounds}), we observe that 
		\begin{align*}
			\int_\Omega (F'_0(\phi_1) - F'_0(\phi_2)) \phi(t) \, \de x 
			= \int_\Omega \left( \int_0^1 F_0''(z \phi_1 + (1-z) \phi_2) \, \de z \right) \phi(t)^2 \, \de x \ge c_0 \norm{\phi(t)}^2_H.
		\end{align*} 
		Hence, from \eqref{eq:contdep:est1int} we further deduce that 
		\begin{align*}
			& \min\left\{\mezzo, c_0\right\} \norm{\phi(t)}^2_V 
			+ \frac14 \int_0^t \norm{\nabla \mu}^2_H \, \de s 
			+ \mezzo \int_0^t \norm{\Delta \phi}^2_H \, \de s \\
			& \quad \le C \int_0^T \underbrace{\left( 1 + \norm{\phi_1}^2_{\Hx2} + \norm{\phi_2}^2_{\Hx2} + \norm{\partial_t \phi_2}_H \right)}_{\in \, \Lt1} \norm{\phi}^2_V \, \de s + CT \norm{\lambda}^2_H.
		\end{align*}
		Therefore, Gronwall's Lemma yields
		\begin{equation}
			\label{eq:contdep:energyest}
			\norm{\phi}^2_{\LT \infty V \cap \LT 2 W} + \norm{\nabla \mu}^2_{\LT 2 H} 
			\le C \norm{\lambda}^2_H.
		\end{equation}
		Moreover, by comparison in \eqref{eq:mu2}, it is now easy to also deduce that 
		\begin{equation}
			\label{eq:contdep:mul2v}
			\norm{\mu}^2_{\LT 2 V} \le C \norm{\lambda}^2_H.
		\end{equation}
		Likewise, by \eqref{eq:phi2}, one can also easily show that 
		\begin{equation*}
			\norm{\phi}^2_{\HT 1 {V^*}} \le C \norm{\lambda}^2_H.
		\end{equation*}
        }
        
		\noindent
		\textsc{Second estimate.}
		We now multiply \eqref{eq:phi2} by $\partial_t \phi$. Then we take the gradient of \eqref{eq:mu2} and we multiply it by $\nabla \partial_t \phi$. 
        Observe that this makes sense in our regularity setting. Integrating the resulting equations over $\Omega$, summing them up, and integrating by parts, we obtain that
		\begin{equation}
			\label{eq:contdep:est2}
			\begin{split}
			& \mezzo \ddt \norm{\Delta \phi}^2_H + \norm{\partial_t \phi}^2_H \\
			& \quad = (\nabla (F'(\phi_1) - F'(\phi_2)), \nabla \partial_t \phi)_H + (\lambda f, \partial_t \phi)_H - (\lambda_1 \phi, \partial_t \phi)_H - (\lambda \phi_2, \partial_t \phi)_H.
			\end{split}   
		\end{equation}
		To estimate the first term on the right-hand side of \eqref{eq:contdep:est2}, we initially rewrite it by integrating by parts and by adding and subtracting some terms (see the first two equalities below).
        We stress that the integration by parts does not produce boundary terms, since the separation property and \eqref{eq:bc} ensure that $\partial_{\n} F'(\phi_i) = F''(\phi_i) \partial_n \phi_i = 0$, for $i=1,2$.
		Then, we employ H\"older and Young's inequalities, Sobolev embeddings and Remark \ref{rmk:Fderivatives} to deduce that
		{\allowdisplaybreaks
		\begin{align*}
			& (\nabla (F'(\phi_1) - F'(\phi_2)), \nabla \partial_t \phi)_H \\
			& \quad = - (\Delta (F'(\phi_1) - F'(\phi_2)), \partial_t \phi)_H \\
			& \quad = - (F''(\phi_1) \Delta \phi, \partial_t \phi)_H 
			- (F'''(\phi_1) (\nabla \phi_1 + \nabla \phi_2) \cdot \nabla \phi, \partial_t \phi)_H \\
			& \qquad - ((F'''(\phi_1) - F'''(\phi_2))\nabla \phi_2 \cdot \nabla \phi_2, \partial_t \phi)_H
			- ((F''(\phi_1) - F''(\phi_2)) \Delta \phi_2, \partial_t \phi)_H \\
			& \quad \le \norm{F''(\phi_1)}_{\Lx\infty} \norm{\Delta \phi}_H \norm{\partial_t \phi}_H \\
			& \qquad + \norm{F'''(\phi_1)}_{\Lx\infty} \left( \norm{\nabla \phi_1}_{\Lx4} + \norm{\nabla \phi_2}_{\Lx4} \right) \norm{\nabla \phi}_{\Lx4} \norm{\partial_t \phi}_H \\
			& \qquad + \norm{F'''(\phi_1) - F'''(\phi_2)}_{\Lx\infty} \norm{\nabla \phi_2}^2_{\Lx4} \norm{\partial_t \phi}_H \\
			& \qquad + \norm{F''(\phi_1) - F''(\phi_2)}_{\Lx\infty} \norm{\Delta \phi_2}_H \norm{\partial_t \phi}_H \\
			& \quad \le C \norm{\Delta \phi}_H \norm{\partial_t \phi}_H 
			+ C \left( \norm{\phi_1}_W + \norm{\phi_2}_W \right) \norm{\phi}_W \norm{\partial_t \phi}_H \\
			& \qquad + C \norm{\phi}_{\Lx\infty} \left( \norm{\phi_2}_W + \norm{\phi_2}^2_W \right) \norm{\partial_t \phi}_H \\
			& \quad \le \frac18 \norm{\partial_t \phi}^2_H 
			+ C \left( 1 + \norm{\phi_1}^2_W + \norm{\phi_2}^4_W \right) \norm{\phi}^2_W. 
		\end{align*}
		}%
		Here we used the Sobolev embeddings $W \hookrightarrow \Wx{1,4}$ and $W \hookrightarrow \Lx\infty$.
		Moreover, we observe that $\norm{\phi_1}^2_W + \norm{\phi_2}^4_W \in \Lt\infty$ by Theorem \ref{thm:strongsols}.
		In particular, the above estimate entails, as a byproduct, that
		\begin{equation}
			\label{eq:contdep:lapldiff}
			\norm{\Delta (F'(\phi_1) - F'(\phi_2))}^2_H \le C \left( 1 + \norm{\phi_1}^2_W + \norm{\phi_2}^4_W \right) \norm{\phi}^2_W.
		\end{equation}
		The remaining terms on the right-hand side of \eqref{eq:contdep:est2} can be easily estimated as follows
		\begin{align*}
			& (\lambda f, \partial_t \phi)_H 
			\le \norm{\lambda}_H \norm{f}_{\Lx\infty} \norm{\partial_t \phi}_H 
			\le \frac18 \norm{\partial_t \phi}^2_H + C \norm{\lambda}^2_H, \\
			& (\lambda_1 \phi, \partial_t \phi)_H 
			\le \norm{\lambda_1}_{\Lx\infty} \norm{\phi}_H \norm{\partial_t \phi}_H 
			\le \frac18 \norm{\partial_t \phi}^2_H + C \norm{\phi}^2_H, \\
			& (\lambda \phi_2, \partial_t \phi)_H 
			\le \norm{\lambda}_H \norm{\phi_2}_{\Lx\infty} \norm{\partial_t \phi}_H 
			\le \frac18 \norm{\partial_t \phi}^2_H + C \norm{\lambda}^2_H.
		\end{align*}
		Thus, by integrating \eqref{eq:contdep:est2} in $(0,t)$, for any given $t \in (0,T)$, and by exploiting the estimates above, we infer that 
		\begin{align*}
			& \mezzo \norm{\Delta \phi(t)}^2_H + \mezzo \int_0^t \norm{\partial_t \phi}^2_H \, \de s \\
			& \quad \le C \int_0^T \left( 1 + \norm{\phi_1}^2_W + \norm{\phi_2}^4_W \right) \left( \norm{\Delta \phi}^2_H + \norm{\phi}^2_H \right) \, \de s
			+ CT \norm{\lambda}^2_H.
		\end{align*} 
		Using again Gronwall's Lemma and recalling \eqref{eq:contdep:energyest} for $\phi$ in $\LT \infty H$, we get
		\begin{equation}
			\label{eq:contdep:phistrong}
			\norm{\phi}^2_{\HT 1 H \cap \LT \infty W} \le C \norm{\lambda}^2_H.
		\end{equation}
		Next, by comparison in \eqref{eq:phi2}, we infer that $\Delta \mu$ is uniformly bounded in $\LT 2 H$ by \eqref{eq:contdep:phistrong}. This means that 
		\begin{equation}
			\label{eq:contdep:mul2w}
			\norm{\mu}^2_{\LT 2 W} \le C \norm{\lambda}^2_H.
		\end{equation}
		Applying now the Laplacian to \eqref{eq:mu2}, by comparison, we infer that 
		\[
			\norm{\Delta^2 \phi}^2_H \le C \norm{\Delta \mu}^2_H + C \norm{\Delta (F'(\phi_1) - F'(\phi_2))}^2_H.
		\]
		Then, using  \eqref{eq:contdep:lapldiff}, \eqref{eq:contdep:phistrong}, \eqref{eq:contdep:mul2w} and recalling that $\norm{\phi_1}^2_W + \norm{\phi_2}^4_W \in \Lt\infty$ (see Theorem \ref{thm:strongsols}), we deduce that 
		\[
			\norm{\Delta^2 \phi}^2_{\LT 2 H} \le C \norm{\lambda}^2_H. 
		\]
		In turn, on account of \eqref{eq:bc2}, by standard elliptic regularity we infer that
		\[
			\norm{\phi}^2_{\LT 2 {\Hx4}} \le C \norm{\lambda}^2_H.
		\]
		This concludes the proof.
	\end{proof}

	\section{Existence of an optimal control}
    \label{EOC}

	Given an initial guess $\phi_0$ satisfying \ref{ass:initial2}, we now consider the optimal control problem:

	\bigskip
	\noindent(CP) \textit{Minimise the cost functional}
	\begin{equation*}
			\mathcal{J}(\phi, \lambda_0) = \frac{\alpha_1}{2} \int_0^T \int_{\Omega \setminus D} |\phi - f|^2 \,\de x \, \de t + \frac{\alpha_2}{2} \int_{\Omega \setminus D} |\phi(T) - f|^2 \,\de x + \frac{\beta}{2} \int_{\Omega \setminus D} \, \left(\frac{1}{\lambda_0} \right)^2 \,\de x,
	\end{equation*}
	\textit{subject to the control constraints}
	\begin{equation*} 
			\lambda_0 \in \Uad := \{ \lambda_0 \in L^{\infty}(\Omega \setminus D) \mid 0 < \lambda_{\text{min}} \le \lambda_0 \le \lambda_{\text{max}} \text{ a.e. in } \Omega \},  
	\end{equation*}
	\textit{and to the state system \eqref{eq:phi}--\eqref{eq:ic}}.
	\bigskip

    \noindent
	We recall that the fidelity coefficient takes the form given by \eqref{eq:lambda_gen}, namely
	\[
		\lambda(x) = \begin{cases}
			0 \quad & \hbox{if $x \in D$}, \\
			\lambda_0(x) \quad & \hbox{if $x \in \Omega \setminus D$}.
		\end{cases}
	\]



    \begin{remark}
        We mention that all the analysis we will carry out in the following can be extended in a straight-forward way to the more general case in which the last term in the cost functional $\Jcal$ is replaced by the integral of a $C^2$ convex function of $\lambda_0$. For instance, to penalise large values of $\lambda_0$, one can choose
        \[
            \frac{\beta}{r} \int_{\Omega \setminus D} \, \left(\frac{1}{\lambda_0} \right)^r \,\de x,
        \]
        for any $r > 0$. 
        Here we just use $r = 2$ for convenience.
    \end{remark}

    \noindent
	Our assumptions on the parameters at play are:
	\begin{enumerate}[font = \bfseries, label = C\arabic*., ref = \bf{C\arabic*}]
		\item\label{ass:c1} $\alpha_1, \alpha_2, \beta \ge 0$, but not all equal to $0$.
		\item\label{ass:c2} $\lambda_{\text{min}}, \lambda_{\text{max}} \in \R$ such that $0 < \lambda_{\text{min}} < \lambda_{\text{max}}$.
	\end{enumerate}	
	First, we prove the existence of an optimal control.

	\begin{theorem}
		\label{thm:excontrol}
		Assume hypotheses \ref{ass:omega}--\ref{ass:f}, \ref{ass:initial2} and \ref{ass:c1}--\ref{ass:c2}.
		Then, the optimal control problem (CP) admits at least one solution $\bar{\lambda}_0 \in \Uad$, that is, if $(\phib, \mub)$ is the solution to \eqref{eq:phi}--\eqref{eq:ic} associated to $\lambdazb$, then  
		\[
			\Jcal(\phib, \lambdazb) = \min_{\lambda_0 \in \Uad} \Jcal(\phi, \lambda_0).
		\]
	\end{theorem}

	\begin{proof}
		Let $\{\lambda_0^n \}_{n \in \N} \subset \Uad$ be a minimising sequence such that 
		\[
			0 \le \inf_{\lambda_0 \in \Uad} \Jcal(\phi, \lambda_0) = \lim_{n \to + \infty} \Jcal(\phi_n, \lambda_0^n),
		\]
		where $(\phi_n, \mu_n)$ is the solution to \eqref{eq:phi}--\eqref{eq:ic} corresponding to $\lambda_0^n$, with the regularity given by Theorem \ref{thm:strongsols}.
		Observe that $\{\lambda_0^n\}_{n \in \N} \subset \Uad$ is bounded in $\Lx\infty$. 
		Thus, by Banach--Alaoglu's theorem, we can infer that there exists $\lambdazb \in \Lx\infty$ such that, up to a subsequence,
		\[
			\lambda_0^n \weakstar \lambdazb \hbox{ weakly star in $\Lx\infty$} \quad \text{and} \quad \lambda_0^n \weak \lambdazb \hbox{ weakly in $\Lx p$ for any $p \in (1, + \infty)$}.
		\]
		Moreover, since $\Uad$ is convex and closed in $\Lx p$ for any $p \in (1, + \infty)$, which means that $\Uad$ is also weakly closed in $\Lx p$ for any $p \in (1, +\infty)$, it follows that $\lambdazb \in \Uad$.  

		Next, consider the solutions $(\phi_n, \mu_n)$ corresponding to $\lambda_0^n$.
		By Theorem \ref{thm:strongsols}, we have that they are uniformly bounded in the spaces of strong solutions. 
		Thus, using Banach--Alaoglu's theorem once more, we can deduce that, up to a subsequence, 
		\begin{align*}
			& \phi_n \weakstar \phib \quad \hbox{weakly star in $\HT 1 V \cap \LT \infty {\Wx{2,p}}$ for any $p \ge 2$,} \\
			& \mu_n \weakstar \mub \quad \hbox{weakly star in $\LT \infty V \cap \LT 2 W$.}
		\end{align*}
		In particular, thanks to well-known compact embeddings (see \cite[Section 8, Corollary 4]{S1986}), it follows that $\phi_n \to \phib$ strongly in $\CT 0 {\Hx s}$ for any $0 < s < 2$ which implies that $\phi_n \to \phib$ strongly in $C^0(\bar{Q_T})$, since $\Hx {s} \hookrightarrow  C^0(\bar{Q_T})$ if $s > 1$ in dimension two.
		Moreover, by uniform convergence, $\phib$ still satisfies the strict separation property, namely there exists $\delta > 0$ such that $\norm{\phib}_{C^0(\bar{Q_T})} \le 1 - \delta$. 
		This fact, by the continuity of $F'$ inside $(-1,1)$, implies that 
		\[
			F'(\phi_n) \to F'(\phib) \quad \hbox{strongly in $C^0(\bar{Q_T}) $.}
		\]
		Therefore, by exploiting the weak and strong convergences written above, it is now easy to show that the limit $(\phib, \mub)$ also satisfies the system \eqref{eq:phi}--\eqref{eq:ic} with respect to $\lambdazb$.
		Then, we can deduce that 
		\[
			\inf_{\lambda_0 \in \Uad} \Jcal(\phi, \lambda_0) \le \Jcal(\phib, \lambdazb).
		\]

		Finally, if we think $\Jcal(\phi, \lambda_0)$ as a functional $\widetilde{\Jcal}(\phi, \phi(T), \lambda_0)$ of three variables, we observe that $\widetilde{\Jcal}$ is weakly lower semi-continuous in $\Lqt2 \times H \times \{ \lambda_0 \in \Lxod2 \mid \lambda_0 \ge \lambda_* > 0 \}$, for some $0 < \lambda_* < \lambda_{\text{min}}$, with respect to the weak topology in $L^2$. 
		Indeed, it is easy to see that $\widetilde{\Jcal}$ is strongly continuous and convex in a set which is also closed and convex itself in $\Lqt2 \times H \times \Lxod2$. 
		Thus, a consequence of Hahn--Banach's theorem, implies that $\widetilde{\Jcal}$ is weakly lower semi-continuous. 
		Hence, we can infer that 
		\[
			\Jcal(\phib, \lambdazb) \le \liminf_{n \to + \infty} \Jcal(\phi_n, \lambda_0^n) = \inf_{\lambda_0 \in \Uad} \Jcal(\phi, \lambda_0) \le \Jcal(\phib, \lambdazb),
		\]
		which means that $\lambdazb$ is optimal.
		The proof is complete.
	\end{proof}

    \section{First-order optimality conditions}
    \label{OC1}

    We now introduce the control-to-state operator $\Scal$, associating to any $\lambda_0 \in \Uad$ the corresponding unique strong solution $(\phi, \mu) = \Scal(\lambda_0)$ to \eqref{eq:phi}--\eqref{eq:ic}.
	For the sake of future convenience, we set
	\begin{align*}
		& \XX := (\HT 1 V \cap \LT \infty {\Hx 3} \cap \LT 2 {\Hx4}) \times (\LT \infty V \cap \LT 2 W), \\
		& \YY := (\HT 1 H \cap \LT \infty W \cap \LT 2 {\Hx4}) \times \LT 2 W.
	\end{align*}
	Then, by Theorem \ref{thm:strongsols}, we have that $\Scal: \Uad \to \XX$  is well defined and that, by Theorem \ref{thm:contdep}, $\Scal: \Uad \to \YY$ is Lipschitz continuous with respect to the $H$-norm in $\Uad$ and the natural one in $\YY$.
	Our next step is to show the Fr\'echet differentiability of $\Scal$ in order to differentiate the cost functional and derive the first-order necessary optimality conditions.
    Indeed, in this section, we first study the linearised system and we prove the Fr\'echet differentiability of $\Scal$. Then, we prove the well-posedness of the adjoint system which allows us to show the first-order optimality conditions. 

	\subsection{Linearised system}

	As an \emph{ansatz} for the Fr\'echet derivative of the control-to-state operator, we now introduce the linearised system corresponding to our control problem.
	Indeed, to formulate the linearised system, we fix an optimal control $\lambdazb \in \Uad$, corresponding to an optimal state $(\phib, \mub) \in \XX$, and we linearise near $\lambdazb$:
	\[
		\phi = \phib + \xi, \quad \mu = \mub + \eta, \quad \lambda_0 = \lambdazb + h_0,
	\]
	with $h_0 \in L^\infty(\Omega \setminus D)$ such that $\norm{h_0}_{L^\infty(\Omega \setminus D)} < \lambda_{\text{min}}$ and $\lambdazb + h_0 \in \Uad$. 
	Then, by approximating the non-linearities at the first order of their Taylor expansion, it is easy to see that $(\xi, \eta)$ formally satisfies the problem
	\begin{alignat}{2}
		& \partial_t \xi - \Delta \eta = h (f - \phib) - \lambdab \xi \quad && \hbox{in $Q_T$,} \label{eq:xi} \\
		& \eta = - \Delta \xi + F''(\phib) \xi \quad && \hbox{in $Q_T$,} \label{eq:eta} \\
		& \partial_{\n} \xi = \partial_{\n} \eta = 0 \quad && \hbox{on $\Sigma_T$,} \label{eq:bclin} \\
		& \xi(0) = 0 \quad && \hbox{in $\Omega$,} \label{eq:iclin}
	\end{alignat}
	where we set
	\begin{equation}
		\label{eq:lambdab}
		\lambdab(x) = \begin{cases}
			0 & \hbox{if $x \in D$,} \\
			\lambdazb(x) & \hbox{if $x \in \Omega \setminus D$,}
		\end{cases}
		\quad \text{and} \quad
		h(x) = \begin{cases}
			0 & \hbox{if $x \in D$,} \\
			h_0(x) & \hbox{if $x \in \Omega \setminus D$.}
		\end{cases}
	\end{equation}
	We now prove the well-posedness of the linearised problem \eqref{eq:xi}--\eqref{eq:iclin} in the more general case in which $h \in H$. 

	\begin{proposition}
		\label{prop:linearised}
		Let \ref{ass:omega}--\ref{ass:f}, \ref{ass:initial2} and \ref{ass:c1}--\ref{ass:c2} hold. 
		If $(\phib, \mub) \in \XX$ be the strong solution to \eqref{eq:phi}--\eqref{eq:ic} corresponding to $\lambdazb \in \Uad$, 
		then, for any $h \in H$ the linearised problem \eqref{eq:xi}--\eqref{eq:iclin} admits a unique (weak) solution $(\xi,\eta)$  such that 
		\begin{align*}
			& \xi \in \HT 1 {V^*} \cap \LT \infty V \cap \LT 2 W, \\
			& \eta \in \LT 2 V,
		\end{align*}
		and
		\begin{alignat*}{2}
			&\duality{\partial_t \xi, v}_V + (\nabla \eta, \nabla v)_H = (h(f - \phib), v)_H - (\lambdab \xi, v)_H \quad && \forall v \in V, \hbox{ a.e. in } (0,T),\\
			& (\eta, v)_H = (\nabla \xi, \nabla v)_H + (F''(\phi) \xi, v)_H \quad && \forall v \in V, \hbox{ a.e. in } (0,T),\\
            &\xi(0)=0.
		\end{alignat*}	
		This weak solution is actually strong. Indeed, there holds
		\begin{align*}
			& \xi \in \HT 1 H \cap \LT \infty W \cap \LT 2 {\Hx4}, \\
			& \eta \in \LT 2 W,
		\end{align*}
		and we have the uniform estimate
		\begin{equation}
			\label{eq:lin:result}
			\norm{\xi}^2_{\HT 1 H \cap \LT \infty W \cap \LT 2 {\Hx4}} + \norm{\eta}^2_{\LT 2 W} \le C \norm{h}^2_H.
		\end{equation}
	\end{proposition}

	\begin{proof}
		The proof can be carried out rigorously through a Faedo--Galerkin discretisation scheme using finite-dimensional spaces generated by the eigenvectors of the operator $\mathcal{N}$. 
		However, being the system linear, this procedure is quite standard, thus we proceed by providing only formal \emph{a priori} estimates, leaving the details to the interested reader. 
        In particular, due to the linearity of the system, the discrete solutions are $C^1$ with respect to time with values in the discrete eigenspaces, thus all the upcoming \emph{a priori} estimates make sense within the discretisation scheme.

		\noindent
		\textsc{First estimate.}
		We begin by testing \eqref{eq:xi} with $\eta$ and \eqref{eq:eta} with $- \partial_t \xi + \eta$ in $H$. Summing the obtained equations up, we obtain:
		\begin{equation}
			\label{eq:lin:est1}
			\begin{split}
			& \mezzo \ddt \norm{\nabla \xi}^2_H + \norm{\eta}^2_H + \norm{\nabla \eta}^2_H + (F''(\phib) \xi, \partial_t \xi)_H \\
			& \quad = (h(f - \phib), \eta)_H - (\lambdab \xi, \eta)_H + (F''(\phib) \xi, \eta)_H - (\Delta \xi, \eta)_H.
			\end{split}
		\end{equation} 
	   Using \ref{ass:F} and Leibniz's rule, we get
		\begin{align*}
			(F''(\phib) \xi, \partial_t \xi)_H
			& = (F''_0(\phib) \xi, \partial_t \xi)_H + (F''_1(\phib) \xi, \partial_t \xi)_H \\
			& = \mezzo \ddt (F''_0(\phib) \xi, \xi)_H - (F_0'''(\phib) \partial_t \phib \, \xi, \xi)_H + (F_1''(\phib) \xi, \partial_t \xi)_H \\
			& = \mezzo \ddt (F''_0(\phib) \xi, \xi)_H - I_1 + I_2. 
		\end{align*}
		Next, we estimate $I_1$ by using H\"older's inequality, the Sobolev embedding $V \hookrightarrow \Lx4$ and Remark \ref{rmk:Fderivatives}. This gives
		\begin{equation*}
			I_1  = \int_\Omega F_0'''(\phib) \partial_t \phib \, \xi^2 \, \de x 
			 \le \norm{F_0'''(\phib)}_{\Lx\infty} \norm{\partial_t \phib}_{H} \norm{\xi}^2_{\Lx 4} 
			\le C \norm{\partial_t \phib}_{H} \norm{\xi}^2_V,
		\end{equation*}
		where $\norm{\partial_t \phib}_{H} \in \Lt2$ by Theorem \ref{thm:strongsols}. 
		To estimate $I_2$, we rewrite $\partial_t \xi$ through equation \eqref{eq:xi}, then we integrate by parts and finally we use H\"older and Young's inequalities, Sobolev embeddings and  \ref{ass:F}. Hence, we find
		\begin{align*}
			I_2 & = \int_\Omega F_1''(\phib) \xi \, \partial_t \xi \, \de x \\
			& = \int_\Omega F_1''(\phib) \xi \, \Delta \eta \, \de x 
			+ \int_\Omega F_1''(\phib) \xi \left( h(f - \phib) - \lambdab \xi \right) \, \de x \\
			& = - \int_\Omega F_1''(\phib) \nabla \xi \cdot \nabla \eta \, \de x 
			- \int_\Omega F_1'''(\phib) \nabla \phib \cdot \nabla \eta \, \xi \, \de x 
			+ \int_\Omega F_1''(\phib) \xi \left( h(f - \phib) - \lambdab \xi \right) \, \de x \\
			& \le \norm{F_1''(\phib)}_{\Lx\infty} \norm{\nabla \xi}_H \norm{\nabla \eta}_H + \norm{F_1'''(\phib)}_{\Lx\infty} \norm{\nabla \phib}_{\Lx 4} \norm{\nabla \eta}_H \norm{\xi}_{\Lx 4} \\
			& \quad + \norm{F_1''(\phib)}_{\Lx\infty} \norm{\xi}_H \left( \left( \norm{f}_{\Lx\infty} + \norm{\phib}_{\Lx\infty} \right) \norm{h}_H + \norm{\lambdab}_{\Lx\infty} \norm{\xi}_H \right) \\
			& \le \frac12 \norm{\nabla \eta}^2_H + C \left( 1 + \norm{\nabla \phib}^2_{\Lx 4} \right) \norm{\xi}^2_V + \norm{h}^2_H, 
		\end{align*} 
		where $\norm{\nabla \phib}^2_{\Lx 4} \in \Lt\infty$ owing to Theorem \ref{thm:strongsols} and the Sobolev embedding $\Hx2 \hookrightarrow \Wx{1,4}$.
		Going back to \eqref{eq:lin:est1}, we now estimate the terms on the right-hand side by using similar techniques, namely,
		\begin{align*}
			& (h(f - \phib), \eta)_H 
			\le \norm{h}_H \left( \norm{f}_{\Lx\infty} + \norm{\phib}_{\Lx\infty} \right) \norm{\eta}_H 
			\le \frac14 \norm{\eta}^2_H + C \norm{h}^2_H, \\
			& (\lambdab \xi, \eta)_H \le \norm{\lambdab}_{\Lx\infty} \norm{\xi}_H \norm{\eta}_H \le \frac14 \norm{\eta}^2_H + C \norm{\xi}^2_H, \\
			& (F''(\phib) \xi, \eta)_H \le \norm{F''(\phib)}_{\Lx\infty} \norm{\xi}_H \norm{\eta}_H \le \frac14 \norm{\eta}^2_H + C \norm{\xi}^2_H, \\
			& - (\Delta \xi, \eta)_H = (\nabla \xi, \nabla \eta)_H \le \frac14 \norm{\nabla \eta}^2_H + \norm{\nabla \xi}^2_H.
		\end{align*}
		Therefore, by integrating \eqref{eq:lin:est1} on $(0,t)$, for any $t \in (0,T)$, recalling \eqref{eq:iclin}, and putting together all the previous estimates, we deduce that 
		\begin{align*}
			& \mezzo \left( (F''_0(\phib) \xi(t), \xi(t))_H + \norm{\nabla \xi(t)}^2_H  \right) + \frac14 \int_0^t \norm{\eta}^2_V \, \de s \\
			& \quad \le C \int_0^T \left( 1 + \norm{\partial_t \phib}_H + \norm{\nabla \phib}^2_{\Lx 4} \right) \norm{\xi}^2_V \, \de x + CT \norm{h}^2_H.
		\end{align*} 
		Then, \ref{ass:F} entails that 
		\[
			(F''_0(\phib) \xi(t), \xi(t))_H + \norm{\nabla \xi(t)}^2_H 
			\ge c_0 \norm{\xi(t)}^2_H + \norm{\nabla \xi(t)}^2_H
			\ge \min\{c_0,1\} \norm{\xi(t)}^2_V.
		\]
		Thus, using Gronwall's Lemma, we conclude that 
		\begin{equation}
			\label{eq:lin:energyest}
			\norm{\xi}^2_{\LT \infty V} + \norm{\eta}^2_{\LT 2 V} \le C \norm{h}^2_H.
		\end{equation}
		Moreover, by comparison in \eqref{eq:xi} and \eqref{eq:eta}, it is now easy to see that  
		\begin{equation}
			\label{eq:lin:phil2w}
			\norm{\phi}^2_{\HT 1 {V^*} \cap \LT 2 W} \le C \norm{h}^2_H. 
		\end{equation}
		Using the above bounds, one can simply pass to the limit in the discretisation framework and show that there exists a (weak) solution $(\xi, \eta)$ to the linearised system \eqref{eq:xi}--\eqref{eq:iclin} with the prescribed regularities. 
		Additionally, being the system linear, uniqueness follows from estimate \eqref{eq:lin:energyest}.

		\noindent
		\textsc{Second estimate.}
		We now test \eqref{eq:xi} with $\partial_t \xi$ and \eqref{eq:eta} with $\Delta \partial_t \xi$ and we sum them up. This gives
		\begin{equation}
			\label{eq:lin:est2}
			\mezzo \ddt \norm{\Delta \xi}^2_H + \norm{\partial_t \xi}^2_H = (F''(\phib) \xi, \Delta \partial_t \xi)_H + (h(f - \phib), \partial_t \xi)_H - (\lambdab \xi, \partial_t \xi)_H. 
		\end{equation}
		Note that, to be rigorous, one should perform this estimate within the Galerkin discretisation framework. 
		Nevertheless, our formal \emph{a priori} estimates can easily be obtained also in the corresponding discretised version.
        
		Let us estimate the first term on the right-hand side of \eqref{eq:lin:est2}. We integrate by parts and then employ H\"older and Young's inequalities, the Sobolev embeddings $W \hookrightarrow \Wx{1,4}$ and $W \hookrightarrow \Lx\infty$ and Remark \ref{rmk:Fderivatives} to deduce the following:
		\begin{align*}
			(F''(\phib) \xi, \Delta \partial_t \xi)_H 
			& = (\Delta (F''(\phib) \xi), \partial_t \xi)_H \\
			& = (F''(\phib) \Delta \xi + 2 F'''(\phib) \nabla \phib \cdot \nabla \xi + F^{(4)}(\phib) \nabla \phib \cdot \nabla \phib \, \xi, \partial_t \xi)_H \\
			& \le \norm{F''(\phib)}_{\Lx\infty} \norm{\Delta \xi}_H \norm{\partial_t \xi}_H \\
			& \quad + 2 \norm{F'''(\phib)}_{\Lx\infty} \norm{\nabla \phib}_{\Lx4} \norm{\nabla \xi}_{\Lx4} \norm{\partial_t \xi}_H \\
			& \quad + \norm{F^{(4)}(\phib)}_{\Lx\infty} \norm{\nabla \phib}^2_{\Lx4} \norm{\xi}_{\Lx\infty} \norm{\partial_t \xi}_H \\
			& \le \frac14 \norm{\partial_t \xi}^2_H + C \left( 1 + \norm{\phib}^2_W + \norm{\phib}^4_W \right) \norm{\xi}^2_W. 
		\end{align*}
		On the other hand, the estimate above also implies that 
		\begin{equation}
			\label{eq:lin:lapl}
			\norm{\Delta (F''(\phib) \xi)}^2_H \le C \left( 1 + \norm{\phib}^2_W + \norm{\phib}^4_W \right) \norm{\xi}^2_W,
		\end{equation}
		where $\norm{\phib}^2_W + \norm{\phib}^4_W \in \Lt\infty$ by Theorem \ref{thm:strongsols}. 
		Next, we can also easily estimate the remaining terms on the right-hand side of \eqref{eq:lin:est2}, namely,
		\begin{align*}
			& (h(f - \phib), \partial_t \xi)_H 
			\le \norm{h}_H \left( \norm{f}_{\Lx\infty} + \norm{\phib}_{\Lx\infty} \right) \norm{\partial_t \xi}_H 
			\le \frac14 \norm{\partial_t \xi}^2_H + C \norm{h}^2_H, \\
			& (\lambdab \xi, \partial_t \xi)_H
			\le \norm{\lambdab}_{\Lx\infty} \norm{\xi}_H \norm{\partial_t \xi}_H
			\le \frac14 \norm{\partial_t \xi}^2_H + C \norm{\xi}^2_H.
		\end{align*}
		Then, integrating \eqref{eq:lin:est2} over $(0,t)$, for any $t \in (0,T)$, and using all the above estimates, we infer that 
		\begin{align*}
			& \mezzo \norm{\Delta \xi(t)}^2_H + \frac14 \int_0^t \norm{\partial_t \xi}^2_H \, \de s \\
			& \quad \le C \int_0^T \left( 1 + \norm{\phib}^2_W + \norm{\phib}^4_W \right) \left( \norm{\Delta \xi}^2_H + \norm{\xi}^2_H \right) \, \de s + CT \norm{h}^2_H.
		\end{align*}
		Hence, using Gronwall's Lemma and  \eqref{eq:lin:energyest} for $\xi$, we get that 
		\begin{equation}
			\label{eq:lin:xistrong}
			\norm{\xi}^2_{\HT 1 H \cap \LT \infty W} \le C \norm{h}^2_H. 
		\end{equation}
		Next, by comparison in \eqref{eq:xi}, it follows that $\Delta \eta$ is uniformly bounded in $\LT 2 H$ by \eqref{eq:lin:xistrong}, so that 
		\begin{equation}
			\label{eq:lin:etal2w}
			\norm{\eta}^2_{\LT 2 W} \le C \norm{h}^2_H. 
		\end{equation}
		Finally, taking the Laplacian of \eqref{eq:eta}, by comparison, we deduce that
		\[
			\norm{\Delta^2 \xi}^2_H \le C \norm{\Delta \eta}^2_H + C \norm{\Delta (F''(\phib) \xi)}^2_H.
		\]
		Consequently, by using \eqref{eq:lin:lapl}--\eqref{eq:lin:etal2w}, we obtain that 
		\[
			\norm{\Delta^2 \xi}^2_{\LT 2 H} \le C \norm{h}^2_H.
		\]
		Recalling now \eqref{eq:bclin}, by elliptic regularity theory we eventually find that 
		\[
			\norm{\xi}^2_{\LT 2 {\Hx4}} \le C \norm{h}^2_H.
		\]
		This ends the proof.
	\end{proof}

	\begin{remark}
		\label{rmk:lingen}
		For future reference, we observe that the proof of Proposition \ref{prop:linearised} can be easily extended to a more general linear system of the form:
		\begin{alignat}{2}
			& \partial_t \xi - \Delta \eta = - \lambdab \xi + g_1 \quad && \hbox{in $Q_T$,} \label{eq:xigen} \\
			& \eta = - \Delta \xi + F''(\phib) \xi + g_2 \quad && \hbox{in $Q_T$,} \label{eq:etagen} \\
			& \partial_{\n} \xi = \partial_{\n} \eta = 0 \quad && \hbox{on $\Sigma_T$,} \label{eq:bclingen} \\
			& \xi(0) = 0 \quad && \hbox{in $\Omega$,} \label{eq:iclingen}
		\end{alignat}
		where $g_1$ and $g_2$ are given sources ($g_1 = h (f - \phib)$ and $g_2 = 0$ in \eqref{eq:xi}--\eqref{eq:iclin}). 
		In particular, the following results can be proven:
		\begin{itemize}
			\item If $g_1 \in \LT 2 {V^*}$ and $g_2 \in \LT 2 V$, then there exists a unique weak solution to \eqref{eq:xigen}--\eqref{eq:iclingen} such that 
			\begin{align*}
				& \xi \in \HT 1 {V^*} \cap \LT \infty V \cap \LT 2 W, \\
				& \eta \in \LT 2 V,
			\end{align*}
			and the following estimate holds:
			\begin{equation}
				\label{eq:lingen:estweak}
				\begin{split}
				& \norm{\xi}^2_{\HT 1 {V^*} \cap \LT \infty V \cap \LT 2 W} + \norm{\eta}^2_{\LT 2 V} \\
				& \quad \le C \left( \norm{g_1}^2_{\LT 2 {V^*}} + \norm{g_2}^2_{\LT 2 V} \right).
				\end{split}
			\end{equation}
			\item If $g_1 \in \LT 2 H$ and $g_2 \in \LT 2 W$, then there exists a unique strong solution to \eqref{eq:xigen}--\eqref{eq:iclingen} such that 
			\begin{align*}
				& \xi \in \HT 1 H \cap \LT \infty W \cap \LT 2 {\Hx4}, \\
				& \eta \in \LT 2 W,
			\end{align*}
			and the following estimate holds:
			\begin{equation}
				\label{eq:lingen:eststrong}
				\begin{split}
				& \norm{\xi}^2_{\HT 1 H \cap \LT \infty W \cap \LT 2 {\Hx4}} + \norm{\eta}^2_{\LT 2 W} \\
				& \quad \le C \left( \norm{g_1}^2_{\LT 2 H} + \norm{g_2}^2_{\LT 2 W} \right).
				\end{split}
			\end{equation}
		\end{itemize}
	\end{remark}

	\subsection{Fr\'echet differentiability}

	We now fix an open bounded set $\mathcal{U}_R \subset L^\infty(\Omega \setminus D)$ such that $\Uad \subset \mathcal{U}_R$ and $\lambda_0 > 0$ for any $\lambda_0 \in \mathcal{U}_R$. 
	Note that the continuous dependence estimate of Theorem \ref{thm:contdep} now holds in $\mathcal{U}_R$ with a constant $K$ depending only on $R$ and the fixed data of the problem. 
	Here we want to show that the control-to-state operator $\Scal: \mathcal{U}_R \to \WW$ is Fr\'echet differentiable, where 
	\[
		\WW := (\HT 1 {V^*} \cap \LT \infty V \cap \LT 2 W) \times \LT 2 V.
	\]
	More precisely, the following result holds.

	\begin{theorem}
		\label{thm:frechet}
		Let \ref{ass:omega}--\ref{ass:f}, \ref{ass:initial2} and \ref{ass:c1}--\ref{ass:c2} hold. 
		Then, $\Scal: \mathcal{U}_R \to \WW$ is Fr\'echet differentiable, namely for any $\lambdazb \in \mathcal{U}_R$ there exists a unique Fr\'echet derivative $\D \Scal(\lambdazb) \in \mathcal{L}(L^\infty(\Omega \setminus D), \WW)$ such that 
		\begin{equation}
			\label{eq:fre:defdiff}
			\frac{\norm{\Scal(\lambdazb + h_0) - \Scal(\lambdazb) - \D \Scal(\lambdazb)[h_0]}_{\WW}}{\norm{h_0}_{L^2(\Omega \setminus D)}} \to 0 \quad \text{as } \norm{h_0}_{L^2(\Omega \setminus D)} \to 0.
		\end{equation}
		In particular, for any $h_0 \in L^\infty(\Omega \setminus D)$, the Fr\'echet derivative of $\Scal$ at $\lambdazb$, which we denote by $\D \Scal(\lambdazb)[h_0]$, is defined as the solution $(\xi, \eta)$ to the linearised system \eqref{eq:xi}--\eqref{eq:iclin} corresponding to $(\phib, \mub) = \Scal(\lambdazb)$, with datum $h_0$. In addition, the Fr\'echet derivative $\D \Scal$ is Lipschitz continuous as a function from $\mathcal{U}_R$ to the space $\mathcal{L}(L^\infty(\Omega \setminus D), \WW)$.
		More precisely, for any $\lambdazb^1, \lambdazb^2 \in \mathcal{U}_R$ the following estimate holds:
		\begin{equation}
			\label{eq:fre:c1}
			\norm{\D \Scal(\lambdazb^1) - \D \Scal(\lambdazb^2)}_{\mathcal{L}(L^\infty(\Omega \setminus D), \WW)} \le K \norm{\lambdazb^1 - \lambdazb^2}_{L^2(\Omega \setminus D)},
		\end{equation}
		for some constant $K > 0$ depending only on the parameters of the system.
	\end{theorem}

	\begin{remark}
		\label{rmk:Sext}
		Note that the definition of $\D \Scal(\lambdazb)[h_0]$ as the corresponding solution to the linearised system \eqref{eq:xi}--\eqref{eq:iclin} is well defined for any $h_0 \in L^2(\Omega \setminus D)$, and thus for any $h_0 \in L^\infty(\Omega \setminus D)$, due to Proposition \ref{prop:linearised}. 
		Then, $\D \Scal(\lambdazb)$ is a well defined linear bounded operator, i.e., it belongs to $\mathcal{L}(L^\infty(\Omega \setminus D), \WW)$.
		Observe also that \eqref{eq:fre:defdiff} shows the Fr\'echet differentiability with respect to the $L^2(\Omega \setminus D)$-norm, but, since $L^\infty(\Omega \setminus D) \hookrightarrow L^2(\Omega \setminus D)$, this also implies the differentiability in our reference space.
		Moreover, $\D \Scal(\lambdazb)$ can be extended to a linear bounded linear operator in $\mathcal{L}(L^2(\Omega \setminus D), \WW)$.
	\end{remark}

	\begin{proof}
		It is sufficient to prove the result for any small enough perturbation $h_0$, hence we fix $\gamma > 0$ and consider only perturbations such that 
		\[
			\norm{h_0}_{L^2(\Omega \setminus D)} \le \gamma.
		\]
		Let us fix $\lambdazb \in \mathcal{U}_R$ and $h_0$ as above and consider
		\begin{align*}
			& (\phih, \muh) := \Scal(\lambdazb + h_0), \\
			& (\phib, \mub) := \Scal(\lambdazb), \\
			& (\xi, \eta) \hbox{ solution to \eqref{eq:xi}--\eqref{eq:iclin} corresponding to $h_0$}.
		\end{align*}
		In order to show the Fr\'echet differentiability of $\Scal$, it is enough to show that there exists a constant $C > 0$, depending only on the parameters of the system and possibly on $\gamma$, and an exponent $s > 2$ such that 
		\[
			\norm{(\phih, \muh) - (\phib, \mub) - (\xi, \eta)}^2_{\WW} \le C \norm{h_0}^s_{L^2(\Omega \setminus D)}. 
		\]
		If we introduce the additional variables (see Theorem \ref{thm:strongsols} and Proposition \ref{prop:linearised})
		\begin{align*}
			& \psi := \phih - \phib - \xi \in \HT 1 H \cap \LT \infty W \cap \LT 2 {\Hx4}, \\
			& \zeta := \muh - \mub - \eta \in \LT 2 W,  
		\end{align*}
		then this is equivalent to showing that 
		\begin{equation}
			\label{eq:fre:aim}
			\norm{(\psi, \zeta)}^2_{\WW} \le C \norm{h_0}^s_{\Lx2}. 
		\end{equation}
		We can now rewrite problem \eqref{eq:xi}--\eqref{eq:iclin} in terms of $\psi$ and $\zeta$, namely,
		\begin{alignat}{2}
			& \partial_t \psi - \Delta \zeta = \Lambda^h \quad && \hbox{in $Q_T$,} \label{eq:psi} \\
			& \zeta = - \Delta \psi + F^h \quad && \hbox{in $Q_T$,} \label{eq:zeta} \\
			& \partial_{\n} \psi = \partial_{\n} \zeta = 0 \quad && \hbox{on $\Sigma_T$,} \label{eq:bcfre}\\
			& \psi(0) = 0 \quad && \hbox{in $\Omega$,} \label{eq:icfre}
		\end{alignat}
		where 
		\begin{align*}
			& \Lambda^h := (\lambdab + h) (f - \phih) - \lambdab (f - \phib) - h(f - \phib) + \lambdab \xi, \\
			& F^h := F'(\phi^h) - F'(\phib) - F''(\phib) \xi,
		\end{align*}
		and $\lambdab$ and $h$ are as in \eqref{eq:lambdab}.
		We can rewrite $\Lambda^h$ and $F^h$ as follows
		\begin{equation*}
			\Lambda^h = - \lambdab \psi - h (\phih - \phib), \quad
			F^h = F''(\phib) \psi + R^h (\phih - \phib)^2,
		\end{equation*}
		where 
		\[
			R^h = \int_0^1 (1-z) F'''(\phib + z(\phih - \phib)) \, \de z. 
		\]
		Recalling that $\phih$ and $\phib$ satisfy \eqref{eq:separation} and using the continuity of $F'''$, it is easy to see that there exists a constant $C_\gamma > 0$, possibly depending on $\gamma$, such that 
		\begin{equation}
			\label{eq:fre:Rh1}
			\norm{R^h}_{\Lqt\infty} \le C_\gamma.
		\end{equation}   
		Moreover, one can also see that 
		\begin{equation}
			\label{eq:fre:Rh2}
			\norm{\nabla R^h}_{\Lqt\infty} \le C_\gamma.
		\end{equation}
		Indeed, 
        recalling Theorem \ref{thm:strongsols} and \eqref{derbounds}, we have that
		\begin{align*}
			\norm{\nabla R^h}_{\Lx\infty} 
			& \le \int_0^1 (1 - z) \norm{F^{(4)}(\phib + z(\phih - \phib))}_{\Lx\infty} \norm{\nabla \phib + z (\nabla \phih - \nabla \phib)}_{\Lx\infty} \, \de z \\
			& \le C_\gamma \left( \norm{\phih}_{\Wx{1,\infty}} + \norm{\phib}_{\Wx{1,\infty}} \right) \\
			& \le C_\gamma \left( \norm{\phih}_{\LT \infty {\Wx{2,p}}} + \norm{\phib}_{\LT \infty {\Wx{2,p}}} \right) \le C_\gamma,
		\end{align*}
		which implies \eqref{eq:fre:Rh2}. 

		We now observe that problem \eqref{eq:psi}--\eqref{eq:icfre} has the same formal structure as general linearised problem \eqref{eq:xigen}--\eqref{eq:iclingen} in Remark \ref{rmk:lingen} (with $\xi = \psi$ and $\eta = \zeta$), if we put
		\[
			g_1 := - h (\phih - \phib) \quad \text{and} \quad 
			g_2 := R^h (\phih - \phib)^2.
		\]
		Then, by \eqref{eq:lingen:estweak}, we immediately deduce that
		\begin{equation}
			\label{eq:fre:unifest}
			\begin{split}
			& \norm{\psi}^2_{\HT 1 {V^*} \cap \LT \infty V \cap \LT 2 W} + \norm{\zeta}^2_{\LT 2 V} \\
			& \quad \le C \left( \norm{g_1}^2_{\LT 2 {V^*}} + \norm{g_2}^2_{\LT 2 V} \right).
			\end{split}
		\end{equation}
		Let us estimate the norms on the right-hand side. By applying H\"older's inequality, the Sobolev embedding $V \hookrightarrow \Lx 4$, which implies $\Lx {4/3} \hookrightarrow V^*$, and the continuous dependence estimate \eqref{eq:contdep:result}, we infer that 
		\begin{align*}
			\norm{g_1}^2_{\LT 2 {V^*}} 
			& = \int_0^T \norm{h (\phih - \phib)}^2_{V^*} \, \de t \\
			& \le \int_0^T \norm{h (\phih - \phib)}^2_{\Lx {4/3}} \, \de t \\
			& \le \int_0^T \norm{h}^2_H \norm{\phih - \phib}^2_{\Lx 4} \, \de t \\
			& \le \norm{h}^2_H \norm{\phih - \phib}^2_{\LT 2 V}
			\le C \norm{h}^4_H.
		\end{align*}
		Next, by using H\"older's inequality, the Sobolev embeddings $V \hookrightarrow \Lx 4$ and $W \hookrightarrow \Wx{1,4}$, estimates \eqref{eq:fre:Rh1} and \eqref{eq:fre:Rh2} and again the continuous dependence estimate \eqref{eq:contdep:result}, we deduce that 
		\begin{align*}
			\norm{g_2}^2_{\LT 2 H} 
			& = \int_0^T \norm{R^h (\phih - \phib)^2}^2_H \, \de t \\
			& \le \int_0^T \norm{R^h}^2_{\Lx\infty} \norm{\phih - \phib}^4_{\Lx 4} \, \de t \\
			& \le \norm{R^h}^2_{\Lqt\infty} \int_0^T \norm{\phih - \phib}^4_V \, \de t \\
			& \le C T \norm{\phih - \phib}^4_{\LT \infty V} \le C \norm{h}^4_H, \\
			\norm{\nabla g_2}^2_{\LT 2 H} 
			& = \int_0^T \norm{\nabla R^h \, (\phih - \phib)^2 + 2 R^h (\phih - \phib) \nabla (\phih - \phib)}^2_H \, \de t \\
			& \le C \int_0^T \norm{\nabla R^h}^2_{\Lx\infty} \norm{\phih - \phib}^4_{\Lx 4} \, \de t \\
			& \quad + C \int_0^T \norm{R^h}^2_{\Lx\infty} \norm{\phih - \phib}^2_{\Lx 4} \norm{\nabla (\phih - \phib)}^2_{\Lx 4} \, \de t \\
			& \le C T \norm{\nabla R^h}_{\Lqt\infty} \norm{\phih - \phib}^4_{\LT \infty V} \\
			& \quad + C \norm{R^h}^2_{\Lqt\infty} \norm{\phih - \phib}^2_{\LT \infty V} \int_0^T \norm{\phih - \phib}^2_{\Wx{1,4}} \, \de t \\
			& \le C \norm{\phih - \phib}^4_{\LT \infty V} + C \norm{\phih - \phib}^2_{\LT \infty V} \norm{\phih - \phib}^2_{\LT 2 W} \\
			& \le C \norm{h}^4_H. 
		\end{align*}
		Therefore, we deduce from \eqref{eq:fre:unifest} that 
		\[
			\norm{(\psi, \zeta)}^2_{\WW} \le \norm{h}^4_H, 
		\]
		which is exactly \eqref{eq:fre:aim} with $s = 4 > 2$. 

		We are left to prove that $\D \Scal$ is Lipschitz continuous. This amounts to show that, given $\lambdazb^i \in \mathcal{U}_R$, $i = 1,2$, corresponding to the solutions $(\phib_i, \mub_i)$, $i = 1,2$, we have, for any $h_0 \in L^\infty(\Omega \setminus D)$ with $\norm{h_0}_{L^\infty(\Omega \setminus D)} = 1$, that
		\begin{equation}
			\label{eq:fre:lipaim}
			\norm{\D\Scal(\lambdazb^1)[h_0] - \D\Scal(\lambdazb^2)[h_0]}^2_{\WW} 
			= \norm{(\xi_1, \eta_1) - (\xi_2, \eta_2)}^2_{\WW} \le K \norm{\lambdazb^1 - \lambdazb^2}^2_{L^2(\Omega \setminus D)}, 
		\end{equation} 
		which, by taking the supremum over $h_0 \in L^\infty(\Omega \setminus D)$ with $\norm{h_0}_{L^\infty(\Omega \setminus D)} = 1$, yields \eqref{eq:fre:c1}.
		Here, $(\xi_i, \eta_i)$, $i=1,2$, denotes the corresponding strong solution to the linearised system \eqref{eq:xi}--\eqref{eq:iclin} with $(\phib_i, \mub_i)$, $i = 1,2$, respectively,  with the same $h$. 
		
         Let us consider the problem solved by the differences $\xi := \xi_1 - \xi_2$ and $\eta := \eta_1 - \eta_2$, which formally takes the following form (see \eqref{eq:lambdab})
		\begin{alignat}{2}
			& \partial_t \xi - \Delta \eta = - \lambdab_1 \xi - h (\phib_1 - \phib_2) - (\lambdab_1 - \lambdab_2) \xi_2  \quad && \hbox{in $Q_T$,} \label{eq:xi2} \\
			& \eta = - \Delta \xi + F''(\phib_1) \xi + (F''(\phib_1) - F''(\phib_2)) \xi_2 \quad && \hbox{in $Q_T$,} \label{eq:eta2} \\
			& \partial_{\n} \xi = \partial_{\n} \eta = 0 \quad && \hbox{on $\Sigma_T$,} \label{eq:bclin2} \\
			& \xi(0) = 0 \quad && \hbox{in $\Omega$.} \label{eq:iclin2}
		\end{alignat}
		Observe once more that problem \eqref{eq:xi2}--\eqref{eq:iclin2} has the same formal structure as the general linearised problem \eqref{eq:xigen}--\eqref{eq:iclingen}, provided we set 
		\[
			g_1 = - h (\phib_1 - \phib_2) - (\lambdab_1 - \lambdab_2) \xi_2 
			\quad \text{and} \quad 
			g_2 = (F''(\phib_1) - F''(\phib_2)) \xi_2.
		\]
		Then, by Remark \ref{rmk:lingen}, recalling \eqref{eq:lingen:estweak}, we have
		\begin{equation}
		\label{eq:fre:mainestlip}
			\norm{(\xi, \eta)}^2_{\WW} \le C \left( \norm{g_1}^2_{\LT 2 {V^*}} + \norm{g_2}^2_{\LT 2 V} \right).
		\end{equation} 
		Let us estimate the norm of $g_1$ first. Using that $\norm{h}_{\Lx\infty} = 1$ and $\xi_2 \in \LT \infty {V}$ uniformly by Proposition \eqref{prop:linearised}, the Sobolev embedding $V \hookrightarrow \Lx4$, which implies $\Lx{4/3} \hookrightarrow V^*$, and the continuous dependence estimate \eqref{eq:contdep:result}, we have
		\begin{align*}
			\norm{g_1}^2_{\LT 2 {V^*}} & = \int_0^T \norm{g_1}^2_{V^*} \, \de t \\
			& \le C \int_0^T \norm{h (\phib_1 - \phib_2)}^2_{\Lx{4/3}} \, \de t + C \int_0^T \norm{(\lambdab_1 - \lambdab_2) \xi_2}^2_{\Lx{4/3}} \, \de t \\
			& \le C \norm{h}^2_H \int_0^T \norm{\phib_1 - \phib_2}^2_{\Lx4} \, \de t + C \int_0^T \norm{\xi_2}^2_{\Lx4} \norm{\lambdab_1 - \lambdab_2}^2_H \, \de t \\
			& \le C \int_0^T \norm{\phib_1 - \phib_2}^2_V \, \de t 
			+ C T \norm{\xi_2}^2_{\LT \infty V} \norm{\lambdab_1 - \lambdab_2}^2_H \\
			& \le C \norm{\lambdab_1 - \lambdab_2}^2_H.
		\end{align*}
		Concerning the norm of $g_2$, by employing Remark \ref{rmk:Fderivatives}, the Sobolev embedding $W \hookrightarrow \Lx \infty$, the regularity of $\xi_2$ given by Proposition \ref{prop:linearised} and the continuous dependence estimate \eqref{eq:contdep:result}, we find that
		\begin{align*}
			\norm{g_2}^2_{\LT 2 H} 
			& = \int_0^T \norm{(F''(\phib_1) - F''(\phib_2)) \xi_2}^2_H \, \de t \\
			& \le \int_0^T \norm{\xi_2}^2_{\Lx\infty} \norm{F''(\phib_1) - F''(\phib_2)}^2_H \, \de t \\
			& \le \norm{\xi_2}^2_{\LT \infty W} \int_0^T C \norm{\phib_1 - \phib_2}^2_H \, \de t \\
			& \le C \norm{\lambdab_1 - \lambdab_2}^2_H.  
		\end{align*}
		Arguing similarly for $\nabla g_2$, we get
        {\allowdisplaybreaks
		\begin{align*}
			& \norm{\nabla g_2}^2_{\LT 2 H} \\
			& \quad \le C \int_0^T \norm{F'''(\phib_1) \nabla(\phib_1 - \phib_2)}^2_H \, \de t + C \int_0^T \norm{(F'''(\phib_1) - F'''(\phib_2)) \nabla \phib_2 \, \xi_2}^2_H \, \de t \\
			& \qquad + C \int_0^T \norm{(F''(\phib_1) - F''(\phib_2)) \nabla \xi_2}^2_H \, \de t \\
			& \quad \le C \int_0^T \norm{F'''(\phib_1)}^2_{\Lx\infty} \norm{\nabla (\phib_1 - \phib_2)}^2_H \, \de t \\
			& \qquad + C \int_0^T \norm{F'''(\phib_1) - F'''(\phib_2)}^2_{\Lx4} \norm{\nabla \phib_2}^2_{\Lx4} \norm{\xi_2}^2_{\Lx\infty} \, \de t \\
			& \qquad + C \int_0^T \norm{F''(\phib_1) - F''(\phib_2)}^2_{\Lx4} \norm{\nabla \xi_2}^2_{\Lx4} \, \de t \\
			& \quad \le C \int_0^T \norm{\phib_1 - \phib_2}^2_V \, \de t 
			+ C \norm{\phib_2}^2_{\LT \infty W} \norm{\xi_2}^2_{\LT \infty W} \int_0^T \norm{\phib_1 - \phib_2}^2_{\Lx4} \, \de t \\
			& \qquad + \norm{\xi_2}^2_{\LT \infty W} \int_0^T \norm{\phib_1 - \phib}^2_{\Lx4} \, \de t \\
			& \quad \le C \int_0^T \norm{\phib_1 - \phib_2}^2_V \, \de t \le C \norm{\lambdab_1 - \lambdab_2}^2_H.  
		\end{align*}
        }%
		In particular, throughout the estimates above, we employed H\"older's inequality, \eqref{derbounds}, the Sobolev embeddings $V \hookrightarrow \Lx4$, $W \hookrightarrow \Lx \infty$ and $W \hookrightarrow \Wx{1,4}$, the regularities of $\phib_2$ and $\xi_2$ given respectively by Theorem \ref{thm:strongsols} and Proposition \ref{prop:linearised}, and finally the continuous dependence estimate \eqref{eq:contdep:result}. 
		Finally, going back to \eqref{eq:fre:mainestlip}, we conclude that 
		\[
			\norm{(\xi, \eta)}^2_{\WW} \le C \norm{\lambdab_1 - \lambdab_2}^2_H,
		\]
		that is, \eqref{eq:fre:lipaim}. 
		This concludes the proof.
	\end{proof}

	\subsection{Adjoint system}

	In order to derive the first-order optimality conditions for (CP) in a form which is suitable for applications, we now introduce the adjoint problem associated with our optimal control problem. 
	Indeed, we fix an optimal state $(\phib, \mub) = \Scal(\lambdazb)$ and we use the formal Lagrangian method with adjoint variables $(p,q)$. 
	Then, we find that the adjoint problem, which is formally solved by these variables, has the following form:
	\begin{alignat}{2}
		& - \partial_t p - \Delta q + \lambdab p + F''(\phib) q = \alpha_1 \chi_{\Omega \setminus D} (\phib - f) \quad && \hbox{in $Q_T$,} \label{eq:p} \\
		& q = - \Delta p \quad && \hbox{in $Q_T$,} \label{eq:q} \\
		& \partial_{\n} p = \partial_{\n} {q} = 0 \quad && \hbox{on $\Sigma_T$,} \label{eq:bcadj} \\
		& p(T) = \alpha_2 \chi_{\Omega \setminus D} (\phib(T) - f) \quad && \hbox{in $\Omega$.} \label{eq:fcadj}
	\end{alignat}
	Here $\chi_{\Omega \setminus D}$ denotes the characteristic function of the set $\Omega \setminus D$, namely $\chi_{\Omega \setminus D}(x) = 1$ if $x \in \Omega \setminus D$ and $\chi_{\Omega \setminus D}(x) = 0$ otherwise, $\lambdab$ is given by \eqref{eq:lambdab} and $\alpha_1$ and $\alpha_2$ are the constants appearing in the definition of $\Jcal$.

	\begin{remark}
	\label{rmk:adjfunct}
		To derive the adjoint problem, it is useful to rewrite the functional $\Jcal$ as
		\[
			\mathcal{J}(\phi, \lambda_0) = \frac{\alpha_1}{2} \int_0^T \int_{\Omega} \chiod |\phi - f|^2 \,\de x \, \de t + \frac{\alpha_2}{2} \int_{\Omega} \chiod |\phi(T) - f|^2 \,\de x + \frac{\beta}{2} \int_{\Omega \setminus D} \, \left(\frac{1}{\lambda_0}\right)^2 \,\de x.
		\]
		In this way, we can get the equations above.
		We also mention that if $\alpha_2 > 0$, due to the low regularity of the final datum $p(T)$, which can be at most in $\Lx \infty$, we are only able to establish the existence of very weak solutions to \eqref{eq:p}--\eqref{eq:fcadj}. 
		This is enough to derive first-order necessary conditions. However, we will need higher regularity for second-order ones.
		For this reason, we will also assume $\alpha_2 = 0$ in the following (see Section \ref{OC2}).
	\end{remark}

	\noindent
	We now state and prove a well-posedness result for the adjoint problem \eqref{eq:p}--\eqref{eq:fcadj}. 

	\begin{proposition}
		\label{prop:adjoint}
		Let \ref{ass:omega}--\ref{ass:f}, \ref{ass:initial2} and \ref{ass:c1}--\ref{ass:c2} hold. 
	Consider the strong solution $(\phib, \mub) \in \XX$ to \eqref{eq:phi}--\eqref{eq:ic} corresponding to $\lambdazb \in \Uad$. 
		Then the adjoint problem \eqref{eq:p}--\eqref{eq:fcadj} admits a unique weak solution such that 
		\begin{align*}
			& p \in \HT 1 {W^*} \cap \LT \infty H \cap \LT 2 W, \\
			& q \in \LT  2 H,
		\end{align*}
		which satisfies the following variational formulation 
		\begin{align}
			& \duality{- \partial_t p, w}_W - (q, \Delta w)_H + (\lambdab p, w)_H + (F''(\phib)q, w)_H = \alpha_1 (\chiod (\phib - f), w)_H, \label{eq:varform:p} \\
			& (q, w)_H = (- \Delta p, w)_H, \label{eq:varform:q}
		\end{align}
		for almost any~$t \in (0,T)$ and for any $w \in W$, and $p(T) = \alpha_2 \chiod (\phib(T) - f)$ in $H$. Moreover, if $\alpha_2 = 0$, we can further say that the unique solution $(p,q)$ is a strong solution to \eqref{eq:p}--\eqref{eq:fcadj}, satisfying all the equations almost everywhere, and enjoying the regularity:
		\begin{align*}
			& p \in \HT 1 H \cap \LT \infty W \cap \LT 2 {\Hx4}, \\
			& q \in \LT 2 W.
		\end{align*}
	\end{proposition}

	\begin{proof}
		Since \eqref{eq:p}--\eqref{eq:fcadj} is a backward linear problem, one can rigorously prove the existence of a unique weak solution through a convenient Faedo--Galerkin discretisation scheme. 
		For the sake of brevity, here we just obtain the formal \emph{a priori} bounds that are needed to pass to the limit in the discretisation scheme. 
		
		\noindent
		\textsc{First estimate.} We take $w = p$ in \eqref{eq:varform:p} and $w = q$ in \eqref{eq:varform:q} and we sum the corresponding equations to obtain:
		\begin{equation}
			\label{eq:adj:est1}
				- \mezzo \ddt \norm{p}^2_H + \norm{q}^2_H
				= (\lambdab p, p)_H + (F''(\phib) q, p)_H + \alpha_1 (\chiod (\phib - f), p)_H. 
		\end{equation}
		Then, we estimate the terms on the right-hand side of \eqref{eq:adj:est1} by means of H\"older and Young's inequalities and Remark \ref{rmk:Fderivatives}. This gives
		\begin{align*}
			& (\lambdab p, p)_H \le \norm{\lambdab}_{\Lx\infty} \norm{p}^2_H \le C \norm{p}^2_H, \\
			& (F''(\phib) q, p)_H \le \norm{F''(\phib)}_{\Lx\infty} \norm{q}_H \norm{p}_H \le \mezzo \norm{q}^2_H + C \norm{p}^2_H, \\
			& \alpha_1 (\chiod (\phib - f), p)_H \le \mezzo \norm{p}^2_H + \mezzo \norm{\alpha_1 \chiod (\phib - f)}^2_H. 
		\end{align*}
		Hence, by integrating \eqref{eq:adj:est1} in $(t,T)$, for any given $t \in (0,T)$, and using the above estimates, we infer that
		\begin{align*}
			& \mezzo \norm{p(t)}^2_H + \mezzo \int_t^T \norm{q}^2_H \, \de s \\
			& \quad \le \norm{\alpha_2 \chiod (\phib(T) - f)}^2_H + C \int_0^T \norm{p}^2_H \, \de s + C \int_0^T \norm{\alpha_1 \chiod (\phib - f)}^2_H \, \de s. 
		\end{align*} 
		Since $\alpha_2 \chiod (\phib(T) - f) \in H$ and $\alpha_1 \chiod (\phib - f) \in \LT 2 H$, an application of Gronwall's Lemma yields 
		\begin{equation}
			\label{eq:adj:firstest}
				\norm{p}^2_{\LT \infty H} + \norm{q}^2_{\LT 2 H} \le C \left( \norm{\alpha_2 \chiod (\phib(T) - f)}^2_H + \norm{\alpha_1 \chiod (\phib - f)}^2_{\LT 2 H} \right).
		\end{equation}
		Moreover, by comparison in \eqref{eq:p} and \eqref{eq:q}, we also easily deduce that 
		\begin{equation}
			\label{eq:adj:pl2w} 
			\norm{p}^2_{\HT 1 {W^*} \cap \LT 2 W} \le C \left( \norm{\alpha_2 \chiod (\phib(T) - f)}^2_H + \norm{\alpha_1 \chiod (\phib - f)}^2_{\LT 2 H} \right).
		\end{equation}
		Estimates \eqref{eq:adj:firstest} and \eqref{eq:adj:pl2w} are then enough to pass to the limit in the discretisation framework and prove the existence of the desired weak solution, which is also unique, the problem being linear.  
		In particular, observe that, due to the embedding $\HT 1 {W^*} \cap \LT 2 W \hookrightarrow \CT 0 H$, the final condition makes sense in $H$. 

		\noindent
		\textsc{Second estimate.} From now on, we assume that $\alpha_2 = 0$, so that the final condition \eqref{eq:fcadj} becomes $p(T) = 0$.
		Taking $w = q$ in \eqref{eq:varform:p} and $w = \partial_t p$ in \eqref{eq:varform:q} and adding together the resulting identities, we find
		\begin{equation}
			\label{eq:adj:est2}
			- \mezzo \ddt \norm{\nabla p}^2_H + \norm{\nabla q}^2_H = - (\lambdab p, q)_H - (F''(\phib) q, q)_H + (\alpha_1 \chiod (\phib - f), q)_H. 
		\end{equation}
		By using H\"older and Young's inequalities, we estimate the terms on the right-hand side of \eqref{eq:adj:est2} in the following way:
		\begin{align*}
			& (\lambdab p, q)_H 
			\le \norm{\lambdab}_{\Lx\infty} \norm{p}_H \norm{q}_H
			\le C \norm{p}^2_H + C \norm{q}^2_H \\
			& (F''(\phib) q, q)_H \le \norm{F''(\phib)}_{\Lx\infty} \norm{q}^2_H \le C \norm{q}^2_H \\
			& (\alpha_1 \chiod (\phib - f), q)_H \le \mezzo \norm{q}^2_H + \mezzo \norm{\alpha_1 \chiod (\phib - f)}^2_H.
		\end{align*}
		Thus, by integrating \eqref{eq:adj:est2} in $(t,T)$, for any given $t \in (0,T)$, and by exploiting the above estimates, we infer that
		\begin{align*}
			\mezzo \norm{\nabla p (t)}^2_H + \int_t^T \norm{\nabla q}^2_H \, \de s \le C \int_0^T \left( \norm{p}^2_H + \norm{q}^2_H \right) \, \de s + C \int_0^T \norm{\alpha_1 \chiod (\phib - f)}^2_H \, \de s.
		\end{align*}
		Therefore, by using the previous estimate \eqref{eq:adj:firstest} on the right-hand side, we conclude that 
		\begin{equation}
			\label{eq:adj:plinfv}
			\norm{p}^2_{\LT \infty V} + \norm{q}^2_{\LT 2 V} \le C \norm{\alpha_1 \chiod (\phib - f)}^2_{\LT 2 H}.
		\end{equation}

		\noindent
		\textsc{Third estimate.} We now take $w = - \partial_t p$ in \eqref{eq:varform:p} and $w = - \Delta \partial_t p$ in \eqref{eq:varform:q}. Again, summing the resulting identities up gives
		\begin{equation}
			\label{eq:adj:est3}
			- \mezzo \ddt \norm{\Delta p}^2_H + \norm{\partial_t p}^2_H 
			= (\lambdab p, \partial_t p)_H + (F''(\phib) q, \partial_t p)_H - (\alpha_1 \chiod (\phib - f), \partial_t p)_H.
		\end{equation}
		Then, arguing as before, we estimate the terms on the right-hand side and we find
		\begin{align*}
			& (\lambdab p, \partial_t p)_H \le \norm{\lambdab}_{\Lx\infty} \norm{p}_H \norm{\partial_t p}_H \le \frac14 \norm{\partial_t p}^2_H + C \norm{p}^2_H \\
			& (F''(\phib) q, \partial_t p)_H \le \norm{F''(\phib)}_{\Lx\infty} \norm{q}_H \norm{\partial_t p}_H \le \frac14 \norm{\partial_t p}^2_H + C \norm{q}^2_H \\
			& (\alpha_1 \chiod (\phib - f), \partial_t p)_H \le \frac14 \norm{\partial_t p}^2_H + \norm{\alpha_1 \chiod (\phib - f)}^2_H. 
		\end{align*}
		Thus, by integrating \eqref{eq:adj:est3} in $(t,T)$, for any given $t \in (0,T)$, and by exploiting the above estimates, we infer that
		\begin{align*}
			\mezzo \norm{\Delta p (t)}^2_H + \frac14 \int_t^T \norm{\partial_t p}^2_H \, \de s \le C \int_0^T \left( \norm{p}^2_H + \norm{q}^2_H \right) \, \de s + C \int_0^T \norm{\alpha_1 \chiod (\phib - f)}^2_H \, \de s.
		\end{align*}
		Hence, by using the previous estimate \eqref{eq:adj:firstest} on the right-hand side, we conclude that 
		\begin{equation}
			\label{eq:adj:ph1h}
			\norm{p}^2_{\HT 1 H \cap \LT \infty W} \le C \norm{\alpha_1 \chiod (\phib - f)}^2_{\LT 2 H}.
		\end{equation}
		Finally, by comparison in \eqref{eq:p}, due to \eqref{eq:adj:ph1h}, we see that $\Delta q$ is now uniformly bounded in $\LT 2 H$, which implies that 
		\begin{equation}
			\label{eq:adj:ql2w}
			\norm{q}^2_{\LT 2 W} \le C \norm{\alpha_1 \chiod (\phib - f)}^2_{\LT 2 H}.
		\end{equation}
		Moreover, due to \eqref{eq:adj:ql2w}, applying the Laplacian to  \eqref{eq:q}, we also see that $\Delta^2 p$ is uniformly bounded in $\LT 2 H$. 
		Then, by standard elliptic regularity theory (see also \eqref{eq:bcadj}), we deduce that 
		\[
			\norm{p}^2_{\LT 2 {\Hx4}} \le C \norm{\alpha_1 \chiod (\phib - f)}^2_{\LT 2 H}.
		\] 
		The proof is complete.
	\end{proof}

	\begin{remark}
	\label{rmk:adjgen}
		For future convenience, we observe that the proof of Proposition \ref{prop:adjoint} can be easily generalised to a more general backward linear problem of the form:
		\begin{alignat}{2}
			& - \partial_t p - \Delta q + \lambdab p + F''(\phib) q = g_3 \quad && \hbox{in $Q_T$,} \label{eq:pgen} \\
			& q = - \Delta p \quad && \hbox{in $Q_T$,} \label{eq:qgen} \\
			& \partial_{\n} p = \partial_{\n} {q} = 0 \quad && \hbox{on $\Sigma_T$,} \label{eq:bcadjgen} \\
			& p(T) = g_4 \quad && \hbox{in $\Omega$.} \label{eq:fcadjgen}
		\end{alignat}
		where $g_3$ and $g_4$ are given sources ($g_3 = \alpha_1 \chiod (\phib - f)$ and $g_4 = \alpha_2 \chiod (\phib(T) - f)$, or $g_4 = 0$ if $\alpha_2 = 0$, in \eqref{eq:p}--\eqref{eq:fcadj}). 
		In particular, we have the following results:
		\begin{itemize}
			\item If $g_3 \in \LT 2 H$ and $g_4 \in H$, then there exists a unique weak solution to \eqref{eq:pgen}--\eqref{eq:fcadjgen} such that 
			\begin{align*}
				& \xi \in \HT 1 {W^*} \cap \LT \infty H \cap \LT 2 W, \\
				& \eta \in \LT 2 H,
			\end{align*}
			and the following estimate holds:
			\begin{equation}
				\label{eq:adjgen:estweak}
				\begin{split}
				& \norm{p}^2_{\HT 1 {W^*} \cap \LT \infty H \cap \LT 2 W} + \norm{q}^2_{\LT 2 H} \\
				& \quad \le C \left( \norm{g_3}^2_{\LT 2 H} + \norm{g_4}^2_H \right).
				\end{split}
			\end{equation}
			\item If $g_2 \in \LT 2 H$ and $g_4 \in W$, then there exists a unique strong solution to \eqref{eq:pgen}--\eqref{eq:fcadjgen} such that 
			\begin{align*}
				& p \in \HT 1 H \cap \LT \infty W \cap \LT 2 {\Hx4}, \\
				& q \in \LT 2 W,
			\end{align*}
			and the following estimate holds:
			\begin{equation}
				\label{eq:adjgen:eststrong}
				\begin{split}
				& \norm{p}^2_{\HT 1 H \cap \LT \infty W \cap \LT 2 {\Hx4}} + \norm{q}^2_{\LT 2 W} \\
				& \quad \le C \left( \norm{g_3}^2_{\LT 2 H} + \norm{g_4}^2_W \right).
				\end{split}
			\end{equation}
		\end{itemize}
	\end{remark}

	\subsection{First-order optimality conditions}

	Having the adjoint system at our disposal, we can now state and prove the main result of this section. 

	\begin{theorem}
		\label{thm:optcond_first}
		Let \ref{ass:omega}--\ref{ass:f}, \ref{ass:initial2} and \ref{ass:c1}--\ref{ass:c2} hold. 
		If $\lambdazb \in \Uad$ is an optimal control for (CP) and  $(\phib, \mub) \in \XX$ is the corresponding optimal state, i.e.~the strong solution to \eqref{eq:phi}--\eqref{eq:ic} with such $\lambdazb$, then the optimal control $\lambdazb$ satisfies the following variational inequality:
		\begin{equation}
			\label{eq:varineq}
			\int_0^T \int_{\Omega\setminus D} \alpha_1 p (f - \phib) (\lambda_0 - \lambdazb) \, \de x \, \de t 
			- \int_{\Omega \setminus D} \frac{\beta}{\lambdazb^3} (\lambda_0 - \lambdazb) \, \de x \ge 0 
			\quad \hbox{for any $\lambda_0 \in \Uad$,}
		\end{equation} 
        where $(p,q)$ is the adjoint variable to $(\phib, \mub)$, i.e.~the solution to the adjoint system \eqref{eq:p}--\eqref{eq:fcadj}. 
	\end{theorem}

	\begin{proof}
		Observe that the cost functional $\Jcal$ is convex and Fr\'echet differentiable in the space $\CT 0 H \times \mathcal{U}_R$, where $\mathcal{U}_R \subset \Lx\infty$ is an open bounded set such that $\Uad \subset \mathcal{U}_R$ and $\lambda_0 > 0$ for any $\lambda_0 \in \mathcal{U}_R$.
		Moreover, thanks to Theorem \ref{thm:frechet}, we know that the solution operator $\Scal$ is Fr\'echet differentiable from $\mathcal{U}_R$ to $\WW$. 
		Consequently, since $\HT 1 {V^*} \cap \LT 2 V \hookrightarrow \CT 0 H$ (cf. the definition of $\WW$), we also have that the operator $\Scal_1$, which selects the first component of $\Scal$, is Fr\'echet differentiable from $\mathcal{U}_R$ to $\CT 0 H$. 
		Therefore, we can consider the reduced cost functional $\Jcal_{\text{red}}: \mathcal{U}_R \to \R$, defined as 
		\[
			\Jcal_{\text{red}}(\lambda_0) = \Jcal(\Scal_1(\lambda_0), \lambda_0),
		\]
		which is Fr\'echet differentiable in $\mathcal{U}_R$ by the chain rule. Thus, we can rewrite the optimal control problem (CP) through the reduced cost functional as the following minimisation problem 
		\[
			\argmin_{\lambda_0 \in \, \Uad} \Jcal_{\text{red}}(\lambda_0).
		\]
		Recalling now that $\Uad$ is convex and $\Jcal_{\text{red}}$ is Fr\'echet differentiable, if $\lambdazb$ is optimal, then it has to satisfy the necessary optimality condition 
		\[
			\D \Jcal_{\text{red}}(\lambdazb)[\lambda_0 - \lambdazb] \ge 0 \quad \hbox{for any $\lambda_0 \in \Uad$.}
		\] 
		By computing explicitly the derivative of $\Jcal_{\text{red}}$, we get that, for any $\lambda_0 \in \Uad$,
		\begin{equation*}
			\int_0^T \int_\Omega \alpha_1 \chiod (\phib - f) \xi \, \de x \, \de t
			+ \int_\Omega \alpha_2 \chiod (\phib(T) - f) \xi(T) \, \de x 
			- \int_{\Omega \setminus D} \frac{\beta}{\lambdazb^3}(\lambda_0 - \lambdazb) \, \de x \ge 0, 
		\end{equation*}
		where $\xi = \D \Scal_1(\phib)[\lambda_0 - \lambdazb]$ is the first component of the solution $(\xi, \eta)$ to the linearised problem \eqref{eq:xi}--\eqref{eq:iclin} corresponding to $(\phib, \mub)$ and $h_0 = \lambda_0 - \lambdazb$.
		Observe that the right-hand side and the final condition of the adjoint problem \eqref{eq:p}--\eqref{eq:fcadj} appear in the above inequality. 
		Thus, by substituing equations \eqref{eq:varform:p} and \eqref{eq:fcadj} into the previous expression and accounting for the low regularity of the adjoint variables, we deduce that for any $\lambda_0 \in \Uad$
		\begin{align*}
			& \int_0^T \duality{- \partial_t p, \xi}_W \, \de t + \int_0^T \int_\Omega \left( - q \Delta \xi + \lambdab p \, \xi + F''(\phib) q \, \xi \right) \, \de x \, \de t \\
			& \quad + \int_\Omega p(T) \xi(T) \, \de x - \int_{\Omega \setminus D} \frac{\beta}{\lambdazb^3}(\lambda_0 - \lambdazb) \, \de x \ge 0,
		\end{align*}
		where we used the boundary condition \eqref{eq:bclin}.
		Integrating by parts in time and using \eqref{eq:iclin}, we obtain that, for any $\lambda_0 \in \Uad$,
		\begin{align*}
			\int_0^T \int_\Omega \left( p \partial_t \xi - q \Delta \xi + \lambdab p \, \xi + F''(\phib) q \, \xi \right) \, \de x \, \de t - \int_{\Omega \setminus D} \frac{\beta}{\lambdazb^3}(\lambda_0 - \lambdazb) \, \de x \ge 0.
		\end{align*}
		Recalling now the adjoint equation \eqref{eq:q}, we integrate by parts, using the boundary conditions \eqref{eq:bclin} and \eqref{eq:bcadj}, and we find
		\begin{align*}
			0 = \int_0^T \int_\Omega ( - q - \Delta p) \eta \, \de x \, \de t = \int_0^T \int_\Omega ( - q \eta - p \Delta \eta ) \, \de x \, \de t. 
		\end{align*} 
		Hence, combining the above identity with the previous inequality and regrouping some terms, we have that, for any $\lambda_0 \in \Uad$,
		\begin{align*}
			& \int_0^T \int_\Omega \left( \partial_t \xi - \Delta \eta + \lambdab \xi \right) p \, \de x \, \de t
			+ \int_0^T \int_\Omega (- \eta - \Delta \xi + F''(\phib) \xi) q \, \de x \, \de t \\
			& \quad - \int_{\Omega \setminus D} \frac{\beta}{\lambdazb^3}(\lambda_0 - \lambdazb) \, \de x \ge 0.
		\end{align*}
		We now notice that the linearised equations \eqref{eq:xi}--\eqref{eq:eta}, up to their source terms, have appeared in the above inequality. 
		Indeed, we can substitute them, recalling that $h_0 = \lambda_0 - \lambdazb$ in our context, and infer that 
		\begin{align*}
			\int_0^T \int_\Omega p (f - \phib) (\lambda - \lambdab) \, \de x \, \de t 
			- \int_{\Omega \setminus D} \frac{\beta}{\lambdazb^3} (\lambda_0 - \lambdazb) \, \de x \ge 0, 
		\end{align*}
        for any $\lambda_0 \in \Uad$. This is exactly inequality \eqref{eq:varineq} and the proof follows.
	\end{proof}

	\section{Second-order optimality conditions}
    \label{OC2}

	Observe that problem \eqref{eq:phi}--\eqref{eq:ic} is nonlinear and so is the control-to-state operator $\Scal$. Therefore, our optimal control problem (CP) is not convex and the first-order conditions given by Theorem \ref{thm:optcond_first} are only necessary and not sufficient. 
	In this section, we derive second-order sufficient conditions to characterise the optimal solutions of (CP). 
    In general, such conditions require more regularity on the solutions to the adjoint system \eqref{eq:p}--\eqref{eq:fcadj}. In view of Remark \ref{rmk:adjgen}, we observe that one could keep $\alpha_2 > 0$ if, for some particular application, the target image $f$ could be replaced by a regularised version $\ftil \in \Hxod2$, for instance through the system \eqref{eq:f}--\eqref{eq:bcf2}. 
    However, it is not clear if this could be relevant from the point of view of applications (e.~g.~numerical approximations), therefore we stick to the standard case $f \in \Lxod\infty$.
    This choice forces us to assume that 
	\begin{enumerate}[font = \bfseries, label = C\arabic*., ref = \bf{C\arabic*}]
		\setcounter{enumi}{2}
		\item\label{ass:c3} $\alpha_2 = 0$.
	\end{enumerate}
    Note that most classical PDE-based inpainting methods impose either a boundary condition on $\Omega \setminus D$ or a penalty term that keeps $\phi$ close to $f$ for all time. Hence, many of the references in the literature can be viewed as the special case $\alpha_2 = 0$.

	\subsection{Control-to-costate operator}

	In view of computing the second-order derivatives of the cost functional $\Jcal$, we now introduce the \emph{control-to-costate operator}. 
	Then, for any $\lambda_0\in\Uad$, we define
	\[
		\Tcal(\lambda_0) = (p,q),
	\]
	where $(p,q)$ is the solution to the adjoint problem \eqref{eq:p}--\eqref{eq:fcadj} with $\lambdazb = \lambda_0$, $(\phib, \mub) = \Scal(\lambda_0)$ and $\alpha_2 = 0$.  Recalling Proposition \ref{prop:adjoint}, we have that
    \[
		\Tcal: \mathcal{U}_R \supset \Uad \to \YY = (\HT 1 H \cap \LT \infty W \cap \LT 2 {\Hx4}) \times \LT 2 W.
	\]

	The rest of this section is devoted to study the differentiability properties of the operator $\Tcal$. First, we establish its Lipschitz continuity.

	\begin{proposition}
		\label{prop:lipcont_costate}
		Assume that \ref{ass:omega}--\ref{ass:f}, \ref{ass:initial2} and \ref{ass:c1}--\ref{ass:c3} hold.
		Let $\lambda_0^1, \lambda_0^2 \in \mathcal{U}_R$ and let $(p_1, q_1)$ and $(p_2, q_2)$ be the corresponding solutions to the adjoint problem \eqref{eq:p}--\eqref{eq:fcadj}. 
		Then, there exists a constant $K > 0$, depending only on $R$ and the parameters of the system, such that 
		\begin{equation}
			\label{eq:lipadj:result}
			\norm{p_1 - p_2}^2_{\HT 1 H \cap \LT \infty W \cap \LT 2 {\Hx4}} + \norm{q_1 - q_2}^2_{\LT 2 W} \le K \norm{\lambda_0^1 - \lambda_0^2}^2_{L^2(\Omega \setminus D)}.
		\end{equation}   
	\end{proposition}

	\begin{proof}
		We set 
		\[
			p := p_1 - p_2, \qquad q := q_1 - q_2,
		\]
		and we define the extensions $\lambda_1$ and $\lambda_2$ of $\lambda_0^1$ and $\lambda_0^2$ (see \eqref{eq:lambdab}). 
		Then, it is straightforward to check that these new variables formally satisfy the following problem:
		\begin{alignat}{2}
			& - \partial_t p - \Delta q + \lambdab_1 p + F''(\phib_1) q \nonumber \\
			& \quad = \alpha_1 \chi_{\Omega \setminus D} (\phib_1 - \phib_2) - (\lambda_1 - \lambda_2) p_2 - (F''(\phib_1) - F''(\phib_2)) q_2 \quad && \hbox{in $Q_T$,} \label{eq:p2} \\
			& q = - \Delta p \quad && \hbox{in $Q_T$,} \label{eq:q2} \\
			& \partial_{\n} p = \partial_{\n} {q} = 0 \quad && \hbox{on $\Sigma_T$,} \label{eq:bcadj2} \\
			& p(T) = 0 \quad && \hbox{in $\Omega$.} \label{eq:fcadj2}
		\end{alignat}
		Observe that this problem has the same formal structure as the general adjoint problem \eqref{eq:pgen}--\eqref{eq:fcadjgen} introduced in Remark \ref{rmk:adjgen}, provided we set
		\[
			g_3 = \alpha_1 \chi_{\Omega \setminus D} (\phib_1 - \phib_2) - (\lambda_1 - \lambda_2) p_2 - (F''(\phib_1) - F''(\phib_2)) q_2 
			\quad \text{and} \quad 
			g_4 = 0.
		\]
		Hence, if $g_3 \in \LT 2 H$, then, by Remark \ref{rmk:adjgen}, estimate \eqref{eq:adjgen:eststrong} holds, namely 
		\[
			\norm{p}^2_{\HT 1 H \cap \LT \infty W \cap \LT 2 {\Hx4}} + \norm{q}^2_{\LT 2 W} \le C \norm{g_3}^2_{\LT 2 H}.
		\]
		Consequently, if we prove that 
		\begin{equation}
			\label{eq:lipadj:aim}
			\norm{g_3}^2_{\LT 2 H} \le C \norm{\lambda_1 - \lambda_2}^2_H,
		\end{equation}
		then \eqref{eq:lipadj:result} immediately follows.
		Indeed, we have that 
		\begin{align*}
			\norm{g_3}^2_{\LT 2 H}
			& = \int_0^T \norm{g_3}^2_H \, \de t \\
			& \le C \int_0^T \norm{\alpha_1 \chiod (\phib_1 - \phib_2)}^2_H \, \de t 
			+ C \int_0^T \norm{(\lambda_1 - \lambda_2) p_2}^2_H \, \de t \\
			& \quad + C \int_0^T \norm{(F''(\phib_1) - F''(\phib_2))q_2}^2_H \, \de t \\
			& \le C \int_0^T \norm{\phib_1 - \phib_2}^2_H \, \de t 
			+ C \int_0^T \norm{p_2}^2_{\Lx\infty} \norm{\lambda_1 - \lambda_2}^2_H \, \de t \\
			& \quad + C \int_0^T \norm{\phib_1 - \phib_2}^2_{\Lx4} \norm{q_2}^2_{\Lx4} \, \de t \\
			& \le C \int_0^T \norm{\phib_1 - \phib_2}^2_H \, \de t + CT \norm{p_2}^2_{\LT \infty W} \norm{\lambda_1 - \lambda_2}^2_H \\
			& \quad + C \norm{\phib_1 - \phib_2}^2_{\LT \infty V} \int_0^T \norm{q_2}^2_V \, \de t \\
			& \le C \norm{\lambda_1 - \lambda_2}^2_H,
		\end{align*}
		where we used H\"older's inequality, the Sobolev embeddings $V \hookrightarrow \Lx4$ and $W \hookrightarrow \Lx\infty$, Remark \ref{rmk:Fderivatives}, the uniform bounds for $p_2 \in \LT \infty W$ and $q_2 \in \LT 2 V$ given by Proposition \ref{prop:adjoint}, and estimate \eqref{eq:contdep:result}. 
		This concludes the proof.
	\end{proof}

	Next, we want to show that $\Tcal$ is also continuously Fr\'echet differentiable. 
	As an \emph{ansatz} for the Fr\'echet derivative of the control-to-costate operator, we introduce a linearised version of the adjoint problem corresponding to our original control problem.
	Indeed, we fix an adjoint state $(\pb, \qb)$ corresponding to $\lambdazb \in \Uad$, and $(\phib, \mub) = \Scal(\lambdazb)$, and we linearise near $\lambdazb$:
	\[
		p = \pb + P, \quad q = \qb + Q, \quad \phi = \phib + \xi, \quad \lambda_0 = \lambdazb + h_0,
	\]
	with $h_0 \in L^\infty(\Omega \setminus D)$ such that $\norm{h_0}_{L^\infty(\Omega \setminus D)} < \lambda_{\text{min}}$. 
	Then, by approximating the non-linear terms at the first order using their Taylor expansions, we see that $(P, Q)$ formally satisfies the problem
	\begin{alignat}{2}
		& - \partial_t P - \Delta Q + \lambdab P + F''(\phib) Q = \alpha_1 \chiod \xi - h \pb - F'''(\phib) \xi \, \qb \quad && \hbox{in $Q_T$,} \label{eq:P} \\
		& Q = - \Delta P \quad && \hbox{in $Q_T$,} \label{eq:Q} \\
		& \partial_{\n} P = \partial_{\n} Q = 0 \quad && \hbox{on $\Sigma_T$,} \label{eq:bcadjlin} \\
		& P(T) = 0 \quad && \hbox{in $\Omega$,} \label{eq:fcadjlin}
	\end{alignat}
	where $\xi$ is the solution to the linearised problem \eqref{eq:xi}--\eqref{eq:iclin} and we define the extensions $\lambdab$ and $h$ of $\lambdazb$ and $h_0$, respectively, as in \eqref{eq:lambdab}.
    
	We now prove the well-posedness of the backward linear problem \eqref{eq:P}--\eqref{eq:fcadjlin} in the more general case in which $h \in \Lx2$ is a given function.

	\begin{proposition}
		\label{prop:adjlinear}
		Let \ref{ass:omega}--\ref{ass:f}, \ref{ass:initial2} and \ref{ass:c1}--\ref{ass:c3} hold.
		Consider the strong solution $(\pb, \qb) \in \YY$ to the adjoint problem \eqref{eq:p}--\eqref{eq:fcadj} that corresponds to $\lambdazb \in \Uad$ and $(\phib, \mub) = \Scal(\lambdazb) \in \XX$.
		Moreover, given $h \in \Lx2$, consider the strong solution $(\xi, \eta) \in \YY$ to the linearised problem \eqref{eq:xi}--\eqref{eq:iclin} corresponding to $\lambdazb$ and $(\phib, \mub)$.

		Then, for any $h \in \Lx2$, problem \eqref{eq:P}--\eqref{eq:fcadjlin} admits a unique strong solution such that 
		\begin{align*}
			& P \in \HT 1 H \cap \LT \infty W \cap \LT 2 {\Hx4}, \\
			& Q \in \LT 2 W,
		\end{align*}
		and the following estimate holds:
		\begin{equation}
			\label{eq:adjlin:result}
			\norm{P}^2_{\HT 1 H \cap \LT \infty W \cap \LT 2 {\Hx4}} + \norm{Q}^2_{\LT 2 W} \le C \norm{h}^2_H.
		\end{equation}   
	\end{proposition}

	\begin{proof}
		We observe that, if we set
		\[
			g_3 =  \alpha_1 \chiod \xi - h \pb - F''(\phib) \xi \, \qb
			\quad \text{and} \quad 
			g_4 = 0,
		\]
		then problem \eqref{eq:P}--\eqref{eq:fcadjlin} has the same formal structure as the general adjoint problem \eqref{eq:pgen}--\eqref{eq:fcadjgen} introduced in Remark \ref{rmk:adjgen}.
		Hence, if $g_3 \in \LT 2 H$, by Remark \ref{rmk:adjgen}, estimate \eqref{eq:adjgen:eststrong} holds, namely, 
		\[
			\norm{P}^2_{\HT 1 H \cap \LT \infty W \cap \LT 2 {\Hx4}} + \norm{Q}^2_{\LT 2 W} \le C \norm{g_3}^2_{\LT 2 H}.
		\]
		Consequently, we just have to prove that
		\begin{equation*}
			\norm{g_3}^2_{\LT 2 H} \le C \norm{h}^2_H
		\end{equation*}
		to readily deduce the wanted estimate \eqref{eq:adjlin:result}.
		Indeed, we have that 
		\begin{align*}
			\norm{g_3}^2_{\LT 2 H}
			& = \int_0^T \norm{g_3}^2_H \, \de t \\
			& \le C \int_0^T \norm{\alpha_1 \chiod \xi}^2_H \, \de t 
			+ C \int_0^T \norm{h \pb}^2_H \, \de t 
			+ C \int_0^T \norm{F'''(\phib) \xi \, \qb}^2_H \, \de t \\
			& \le C \int_0^T \norm{\xi}^2_H \, \de t 
			+ C \int_0^T \norm{\pb}^2_{\Lx\infty} \norm{h}^2_H \, \de t \\
			& \quad + C \int_0^T \norm{F'''(\phib)}^2_{\Lx\infty} \norm{\xi}^2_{\Lx4} \norm{\qb}^2_{\Lx4} \, \de t \\
			& \le C \int_0^T \norm{\xi}^2_H \, \de t + CT \norm{\pb}^2_{\LT \infty W} \norm{h}^2_H 
			+ C \norm{\xi}^2_{\LT \infty V} \int_0^T \norm{\qb}^2_V \, \de t \\
			& \le C \norm{h}^2_H.
		\end{align*}
		Here, we used H\"older's inequality, the Sobolev embeddings $V \hookrightarrow \Lx4$ and $W \hookrightarrow \Lx\infty$, Remark \ref{rmk:Fderivatives}, the uniform bounds $\pb \in \LT \infty W$ and $\qb \in \LT 2 V$ given by Proposition \ref{prop:adjoint}, and   estimate \eqref{eq:lin:result}. 
		This ends the proof.
	\end{proof}

	We are now ready to show that the control-to-costate operator $\Tcal$ is continuously Fr\'echet differentiable. 

	\begin{theorem}
		\label{thm:adjfrechet}
		Let \ref{ass:omega}--\ref{ass:f}, \ref{ass:initial2} and \ref{ass:c1}--\ref{ass:c3} hold. 
		Then, $\Tcal: \mathcal{U}_R \to \YY$ is Fr\'echet differentiable, namely, for any $\lambdazb \in \mathcal{U}_R$, there exists a unique Fr\'echet derivative $\D \Tcal(\lambdazb) \in \mathcal{L}(L^\infty(\Omega \setminus D), \YY)$ such that 
		\begin{equation*}
			\frac{\norm{\Tcal(\lambdazb + h_0) - \Tcal(\lambdazb) - \D \Tcal(\lambdazb)[h_0]}_{\YY}}{\norm{h_0}_{L^2(\Omega \setminus D)}} \to 0 \quad \text{as } \norm{h_0}_{L^2(\Omega \setminus D)} \to 0.
		\end{equation*}
		In particular, for any $h_0 \in L^\infty(\Omega \setminus D)$, the Fr\'echet derivative of $\Tcal$, denoted by $\D \Tcal(\ub)[h_0]$, is defined as the solution $(P, Q)$ to the problem \eqref{eq:P}--\eqref{eq:fcadjlin} corresponding to $(\pb, \qb) = \Tcal(\lambdazb)$, $(\phib, \mub) = \Scal(\lambdazb)$, $(\xi, \eta) = \D\Scal(\lambdazb)[h_0]$ and $h_0$.
		Moreover, $\D \Tcal$ is Lipschitz continuous as a function from $\mathcal{U}_R$ to the space $\mathcal{L}(L^\infty(\Omega \setminus D), \YY)$.
		More precisely, for any $\lambdazb^1, \lambdazb^2 \in \mathcal{U}_R$, the following estimate holds:
		\begin{equation}
			\label{eq:adjfre:c1}
			\norm{\D \Tcal(\lambdazb^1) - \D \Tcal(\lambdazb^2)}_{\mathcal{L}(L^\infty(\Omega \setminus D), \YY)} \le K \norm{\lambdazb^1 - \lambdazb^2}_{L^2(\Omega \setminus D)},
		\end{equation}
		for some constant $K > 0$ depending only on the parameters of the system and on $R$.
	\end{theorem}

	\begin{remark}
		\label{rmk:Text}
		Note that, as before, the definition of $\D \Tcal(\lambdazb)[h_0]$ as the corresponding solution to problem \eqref{eq:P}--\eqref{eq:fcadjlin} makes sense for any $h_0 \in L^2(\Omega \setminus D)$, and thus for any $h_0 \in L^\infty(\Omega \setminus D)$, due to Proposition \ref{prop:adjlinear}. 
		Then, $\D \Tcal(\lambdazb)$ is well defined as a bounded linear operator, that is, it belongs to $\mathcal{L}(L^\infty(\Omega \setminus D), \YY)$.
		Moreover, $\D \Tcal(\lambdazb)$ can also be extended in order to belong to $\mathcal{L}(L^2(\Omega \setminus D), \WW)$.
	\end{remark}

	\begin{proof}
		We follow the general strategy already used in the proof of Theorem \ref{thm:frechet}.
		Indeed, to prove the Fr\'echet differentiability we consider only small perturbations $h_0$ such that
		\[
			\norm{h_0}_{L^2(\Omega \setminus D)} \le \gamma,
		\]
		for some $\gamma > 0$.
		Next, we fix $\lambdazb \in \mathcal{U}_R$ and $h_0$ as above, set
		\begin{equation*}
			(p^h, q^h) := \Tcal(\lambdazb + h_0), \quad
			 (\pb, \qb) := \Tcal(\lambdazb), \end{equation*}
    and indicate by $(P, Q)$ the solution to \eqref{eq:P}--\eqref{eq:fcadjlin} with respect to $h_0$ and $\lambdazb$.
		
		Then, we need to show that there exists a constant $C > 0$ and an exponent $s > 2$ such that 
		\[
			\norm{(p^h, q^h) - (\pb, \qb) - (P, Q)}^2_{\YY} \le C \norm{h_0}^s_{L^2(\Omega \setminus D)}.
		\]
		If we introduce the additional variables (see Propositions \ref{prop:adjoint} and \ref{prop:adjlinear})
		\begin{align*}
			& \pi : = p^h - \pb - P \in \HT 1 H \cap \LT \infty W \cap \LT 2 {\Hx4}, \\
			& \upsilon : = q^h - \qb - Q \in \LT 2 W,
		\end{align*}
		then showing the above inequality is equivalent to show that 
		\begin{equation}
			\label{eq:adjfre:aim}
			\norm{(\pi, \upsilon)}^2_{\YY} \le C \norm{h_0}^s_{L^2(\Omega \setminus D)}.
		\end{equation}
		To such end, we observe that, $(\pi, \upsilon)$ satisfies the following problem:
		\begin{alignat}{2}
			& - \partial_t \pi - \Delta \upsilon + \lambdab \pi + F''(\phib) \upsilon = \alpha_1 \chiod \psi - h (p^h - \pb) - F^h \quad && \hbox{in $Q_T$,} \label{eq:pi} \\
			& \upsilon = - \Delta \pi \quad && \hbox{in $Q_T$,} \label{eq:ups} \\
			& \partial_{\n} \pi = \partial_{\n} \upsilon = 0 \quad && \hbox{on $\Sigma_T$,} \label{eq:bcadjfre} \\
			& \pi(T) = 0 \quad && \hbox{in $\Omega$,} \label{eq:fcadjfre}
		\end{alignat}
		where the extensions $\lambdab$ and $h$ are defined, as usual, like in \eqref{eq:lambdab}, $(\phi^h, \mu^h) = \Scal(\lambdazb + h_0)$, $(\phib, \mub) = \Scal(\lambdazb)$, $(\xi, \eta) = \D \Scal(\lambdazb)[h_0]$, $\psi = \phi^h - \phib - \xi$ is the same variable appearing in \eqref{eq:psi}--\eqref{eq:icfre} (see the proof of Theorem \ref{thm:frechet}), and 
		\[
			F^h := (F''(\phi^h) - F''(\phib)) q^h - F'''(\phib) \xi \qb.
		\] 
		Using now Taylor's expansion with integral remainder, we can rewrite $F^h$ as follows
		\[
			F^h = F'''(\phib) \psi + (F''(\phi^h) - F''(\phib))(q^h - \qb) + R^h (\phi^h - \phib)^2 \qb,
		\]
		where
		\[
			R^h := \int_0^1 (1 - z) F^{(4)}(\phib + z(\phi^h - \phib)) \, \de z.
		\]
		In particular, arguing as in Theorem \ref{thm:frechet}, we find the bound
		\[
			\norm{R^h}_{\Lqt\infty} \le C_\gamma.
		\]
		Moreover, we recall that throughout the proof of Theorem \ref{thm:frechet}, using \eqref{eq:fre:unifest} and subsequent estimates, we showed that 
		\begin{equation}
			\label{eq:adjfre:psi}
			\norm{\psi}^2_{\HT 1 {V^*} \cap \LT \infty V \cap \LT 2 W} \le C \norm{h}^4_H.
		\end{equation}

		We are now ready to perform some estimates on the solution to problem \eqref{eq:pi}--\eqref{eq:fcadjfre}. 
		Observe that this problem has the same formal structure as the general adjoint problem \eqref{eq:pgen}--\eqref{eq:fcadjgen}, provided we set
		\begin{align*}
			g_3 & := \alpha_1 \chiod \psi - h (p^h - \pb) - F^h \\
			& = \alpha_1 \chiod \psi - h (p^h - \pb) - F'''(\phib) \psi - (F''(\phi^h) - F''(\phib))(q^h - \qb) - R^h (\phi^h - \phib)^2 \qb.
		\end{align*} 
		Thus, by Remark \ref{rmk:adjgen}, we immediately deduce that 
		\[
			\norm{(\pi, \upsilon)}^2_{\YY} \le C \norm{g_3}^2_{\LT 2 H}.
		\]
		On the other hand, using H\"older's inequality, \eqref{eq:adjfre:psi}, the continuous dependence results obtained in Theorem \ref{thm:contdep} and Proposition \ref{prop:lipcont_costate}, Remark \ref{rmk:Fderivatives}, the Sobolev embeddings $V \hookrightarrow \Lx4$ and $W \hookrightarrow \Lx\infty$, and the regularity of $\qb$ given by Proposition \ref{prop:adjoint}, we can deduce the following estimate:
		\begin{align*}
			\norm{g_3}^2_{\LT 2 H}
			& \le C \int_0^T \alpha_1 \norm{\chiod}_{\Lx\infty} \norm{\psi}^2_H \, \de t 
			+ C \int_0^T \norm{p^h - \pb}^2_{\Lx\infty} \norm{h}^2_H \, \de t \\
			& \quad + C \int_0^T \norm{F'''(\phib)}^2_{\Lx\infty} \norm{\psi}^2_H \, \de t 
			+ C \int_0^T \norm{F''(\phi^h) - F''(\phib)}^2_{\Lx4} \norm{q^h - \qb}^2_{\Lx4} \, \de t \\
			& \quad + C \int_0^T \norm{R^h}^2_{\Lx\infty} \norm{\phi^h - \phib}^4_{\Lx4} \norm{\qb}^2_{\Lx\infty} \, \de t \\
			& \le C \norm{\psi}^2_{\LT 2 H}
			+ C \norm{h}^2_H \int_0^T \norm{p^h - \pb}^2_W \, \de t \\
			& \quad + C \norm{\phi^h - \phib}^2_{\LT \infty V} \int_0^T \norm{q^h - \qb}^2_V \, \de t 
			+ C \norm{\phi^h - \phib}^4_{\LT \infty V} \int_0^T \norm{\qb}^2_W \, \de t \\
			& \le C_\gamma \norm{h}^4_H, 
		\end{align*}
		which gives
		\[
			\norm{(\pi, \upsilon)}^2_{\YY} \le C \norm{h}^4_H,
		\]
		so that \eqref{eq:adjfre:aim} holds, i.e., the Fr\'echet differentiability of $\Tcal$.

		We now want to prove that $\D \Tcal$ is Lipschitz continuous. 
		To do this, we have to show that, given $\lambdazb^i \in \mathcal{U}_R$, with corresponding solutions $(\phib^i, \mub^i)$ and $(\pb^i, \qb^i)$, $i=1,2$, then for any $h_0 \in L^\infty(\Omega \setminus D)$ with $\norm{h_0}_{L^\infty(\Omega \setminus D)} = 1$ it holds that 
		\begin{equation}
			\label{eq:adjfre:lipaim}
			\norm{\D \Tcal(\lambdazb^1)[h_0] - \D \Tcal (\lambdazb^2)[h_0]}^2_{\YY} = \norm{(P_1, Q_1) - (P_2, Q_2)}^2_{\YY} \le C \norm{\lambdazb^1 - \lambdazb^2}^2_{L^2(\Omega \setminus D)}, 
		\end{equation}
		which implies \eqref{eq:adjfre:c1}. 
		Here, we called $(P_i, Q_i)$, $i=1,2$, the corresponding solutions to \eqref{eq:P}--\eqref{eq:fcadjlin} with $(\pb_i, \qb_i)$ and $(\phib_i, \mub_i)$, $i = 1,2$, respectively and the same $h_0$. 
		We now consider the problem formulated for $P := P_1 - P_2$ and $Q := Q_1 - Q_2$, which takes the following form (see also \eqref{eq:lambdab})
		\begin{alignat}{2}
			& - \partial_t P - \Delta Q + \lambdab_1 P + F''(\phib_1) Q \nonumber \\ 
			& \quad = \alpha_1 \chiod (\xi_1 - \xi_2) - h (\pb_1 - \pb_2) - (\lambdab_1 - \lambdab_2) P_2 - (F''(\phib_1) - F''(\phib_2)) Q_2 \nonumber \\
			& \quad - F'''(\phib_1) \xi_1 (\qb_1 - \qb_2) - F'''(\phib_1)(\xi_1 - \xi_2) \qb_2 - (F'''(\phib_1) - F'''(\phib_2)) \xi_2 \qb_2 \quad && \hbox{in $Q_T$,} \label{eq:P2} \\
			& Q = - \Delta P \quad && \hbox{in $Q_T$,} \label{eq:Q2} \\
			& \partial_{\n} P = \partial_{\n} Q = 0 \quad && \hbox{on $\Sigma_T$,} \label{eq:bcadjlin2} \\
			& P(T) = 0 \quad && \hbox{in $\Omega$,} \label{eq:fcadjlin2}
		\end{alignat}
		Observe that, also in this case, problem \eqref{eq:P2}--\eqref{eq:fcadjlin2} has the same formal structure as the general adjoint problem \eqref{eq:pgen}--\eqref{eq:fcadjgen}, if we set
		\begin{align*}
			g_3 & := \alpha_1 \chiod (\xi_1 - \xi_2) - h (\pb_1 - \pb_2) - (\lambdab_1 - \lambdab_2) P_2 - (F''(\phib_1) - F''(\phib_2)) Q_2 \\
			& \quad - F'''(\phib_1) \xi_1 (\qb_1 - \qb_2) - F'''(\phib_1)(\xi_1 - \xi_2) \qb_2 - (F'''(\phib_1) - F'''(\phib_2)) \xi_2 \qb_2.
		\end{align*} 
		Hence, taking Remark \ref{rmk:adjgen} into account, we immediately deduce that 
		\[
			\norm{(P,Q)}^2_{\YY} \le C \norm{g_3}^2_{\LT 2 H}.
		\]
		Then, employing H\"older's inequality, the Sobolev embeddings $V \hookrightarrow \Lx4$ and $W \hookrightarrow \Lx \infty$, Remark \ref{rmk:Fderivatives}, the uniform bounds for $(\pb_i, \qb_i)$ and $(P_i, Q_i)$ given by Propositions \ref{prop:adjoint} and \ref{prop:adjlinear}, the continuous dependence estimates given by Theorem \ref{thm:contdep} and Proposition \ref{prop:lipcont_costate} and the Lipschitz continuity of $\D \Scal$ by Theorem \ref{thm:frechet}, we find
		{\allowdisplaybreaks
		\begin{align*}
			\norm{g_3}^2_{\LT 2 H}
			& \le C \int_0^T \alpha_1 \norm{\chiod}_{\Lx\infty} \norm{\xi_1 - \xi_2}^2_H \, \de t 
			+ C \int_0^T \norm{\pb_1 - \pb_2}^2_{\Lx\infty} \norm{h}^2_{H} \, \de t \\
			& \quad + C \int_0^T \norm{\lambdab_1 - \lambdab_2}^2_H \norm{P_2}^2_{\Lx\infty} \, \de t 
			+ C \int_0^T \norm{F''(\phib_1) - F''(\phib_2)}^2_{\Lx4} \norm{Q_2}^2_{\Lx4} \, \de t \\
			& \quad + C \int_0^T \norm{F'''(\phi_1)}^2_{\Lx\infty} \norm{\xi_1}^2_{\Lx4} \norm{\qb_1 - \qb_2}^2_{\Lx4} \, \de t \\
			& \quad + C \int_0^T \norm{F'''(\phib_1)}^2_{\Lx\infty} \norm{\xi_1 - \xi_2}^2_{\Lx4} \norm{\qb_2}^2_{\Lx4} \\
			& \quad + C \int_0^T \norm{F'''(\phib_1) - F'''(\phib_2)}^2_{\Lx4} \norm{\xi_2}^2_{\Lx\infty} \norm{\qb_2}^2_{\Lx4} \\
			& \le C \norm{\xi_1 - \xi_2}^2_{\LT 2 H}
			+ C \int_0^T \norm{\pb_1 - \pb_2}^2_{W} \, \de t 
			+ C \norm{\lambdab_1 - \lambdab_2}^2_H \int_0^T \norm{P_2}^2_W \, \de t \\
			& \quad + C \norm{\phib_1 - \phib_2}^2_{\LT \infty V} \int_0^T \norm{Q_2}^2_V \, \de t
			+ C \norm{\xi_1}^2_{\LT \infty V} \int_0^T \norm{\qb_1 - \qb_2}^2_V \, \de t \\
			& \quad + C \norm{\xi_1 - \xi_2}^2_{\LT \infty V} \int_0^T \norm{\qb_2}^2_V \, \de t \\
			& \quad + C \norm{\xi_2}^2_{\LT \infty W} \norm{\phib_1 - \phib_2}^2_{\LT \infty V} \int_0^T \norm{\qb_2}^2_V \, \de t \\
			& \le C \norm{\lambdab_1 - \lambdab_2}^2_H,
 		\end{align*}
		}%
		so that 
        \eqref{eq:adjfre:lipaim} holds
		and the proof is complete.
	\end{proof}

	\subsection{Second-order optimality conditions}

	By Theorems \ref{thm:frechet} and \ref{thm:adjfrechet}, the control-to-state operator $\Scal$ and the control-to-costate operator $\Tcal$ are now both continuously Fr\'echet differentiable in the respective spaces. 
	Then, the reduced cost functional $\Jred: \mathcal{U}_R \to \R$, defined as $\Jred(\lambda_0) = \Jcal (\Scal_1(\lambda_0), \lambda_0)$, where $\Scal_1$ selects the first component of $\Scal$, is twice continuously Fr\'echet differentiable by the chain rule.
	In particular, we can explicitly compute the second order derivatives of the $\mathcal{C}^2$ functional $\Jred$ in terms of the systems introduced above.
	Indeed, on account of \ref{ass:c3}, we have
	\[
		\Jred(\lambda_0) = \frac{\alpha_1}{2} \int_0^T \int_{\Omega \setminus D} \abs{\phi - f}^2 \, \de x \, \de t + \frac{\beta}{2} \int_{\Omega \setminus D} \frac{1}{\lambda_0^2} \, \de x, 
	\] 
	where $\phi = \Scal_1(\lambda_0)$.
	For a fixed $\lambdazb \in \Uad$, Theorem \ref{thm:optcond_first} shows that 
	\[
		\D \Jred(\lambdazb)[h_0] = \int_0^T \int_{\Omega \setminus D} \alpha_1 (\phib - f) \xi \, \de x \, \de t - \int_{\Omega \setminus D} \frac{\beta}{\lambdazb^3} h_0 \, \de x, 
	\]
	where $\phib = \Scal_1(\lambdazb)$ and $\xi = \D \Scal_1(\lambdazb)[h_0]$. 
	By exploiting the adjoint problem \eqref{eq:p}--\eqref{eq:fcadj}, we also know that we can rewrite the above derivative in the form (see Theorem \ref{thm:optcond_first}) 
	\begin{equation}
	\label{eq:Jprime}
		\D \Jred(\lambdazb)[h_0] = \int_0^T \int_{\Omega \setminus D} \pb (f - \phib) h_0 \, \de x \, \de t - \int_{\Omega \setminus D} \frac{\beta}{\lambdazb^3} h_0 \, \de x,
	\end{equation}
	where $\pb = \Tcal_1(\lambdazb)$, $\phib = \Scal_1(\lambdazb)$. 
	Then, thanks to the chain rule, we can compute the second derivative of $\Jred$ at $\lambdazb$ in the directions $h_0, k_0 \in L^\infty(\Omega \setminus D)$ as follows:
	\begin{align*}
		\D^2 \Jred(\lambdazb)[h_0, k_0] & = \int_0^T \int_{\Omega \setminus D} \D T_1(\lambdazb)[k_0] (f - \phib) h_0 \, \de x \, \de t \\
		& \quad + \int_0^T \int_{\Omega \setminus D} \pb \, \D \Scal_1(\lambdazb)[k_0] h_0 \, \de x \, \de t 
		+ \int_{\Omega \setminus D} \frac{3 \beta}{\lambdazb^4} h_0 k_0 \, \de x,
	\end{align*}
	namely,
	\begin{equation}
		\label{eq:Jsecond}
		\D^2 \Jred(\lambdazb)[h_0, k_0] = \int_0^T \int_{\Omega \setminus D} P \, (f - \phib) h_0 \, \de x \, \de t 
		+ \int_0^T \int_{\Omega \setminus D} \pb \, \xi h_0 \, \de x \, \de t 
		+ \int_{\Omega \setminus D} \frac{3 \beta}{\lambdazb^4} h_0 k_0 \, \de x,
	\end{equation}
	where $P := \D \Tcal_1(\lambdazb)[k_0]$ is the solution to the linearised adjoint problem \eqref{eq:P}--\eqref{eq:fcadjlin} with $k_0$, $\phib := \Scal_1(\lambdazb)$ is the solution to \eqref{eq:phi}--\eqref{eq:ic} with $\lambdazb$, $\pb := \Tcal_1(\lambdazb)$ is the solution to the adjoint problem \eqref{eq:p}--\eqref{eq:fcadj} with $\lambdazb$, and $\xi := \D \Scal_1(\lambdazb)[k_0]$ solves the linearised problem \eqref{eq:xi}--\eqref{eq:iclin} with $k_0$.
	Moreover, we observe that, owing to Remarks \ref{rmk:Sext} and \ref{rmk:Text}, the expressions above for the derivatives of $\Jred$ make sense even for $h_0, k_0 \in L^2(\Omega \setminus D)$. 
	Then, for any $\lambdazb \in \mathcal{U}_R$, $\D^2 \Jred(\lambdazb)$ can be extended to a continuous bilinear form from $L^2(\Omega \setminus D)$ to $\R$.
	
	We are now ready to state and prove the second-order sufficient optimality conditions for our optimal control problem (CP). 
	To do this, we follow the approach developed in \cite{CRT2008,CT2012} (cf. also \cite{CTU2000, troltzsch} for a slightly different one).
	Both were later applied also to Cahn--Hilliard-type systems (see \cite{ACGW2024, EK2020, CSS2021_secondorder, ST2024, CFGS2015, CS2024}).
	The idea of sufficient second-order conditions is to show that if $\D^2 \Jred(\lambdazb)$ is positive definite in an optimal point $\lambdazb \in \Uad$, namely
	\[
		\D^2 \Jred(\lambdazb)[h_0, h_0] > 0 \quad \hbox{for any $h_0 \in L^2(\Omega \setminus D) \setminus \{0\}$,}
	\]
	then such a $\lambdazb$ is a strict local minimiser of $\Jred$ on $\Uad$.
	However, imposing the positive definiteness of $\D^2\Jred$ in all directions $h_0$ is generally too restrictive. 
	Indeed, it suffices to require such a condition only in a certain class of critical directions (cf. \cite[Section 4.10]{troltzsch} and \cite{CRT2008, CT2012} for more details on this topic).
    Essentially, due to some additional information on the structure of the optimal controls that can be extracted from the first-order optimality condition and from the fact that $\Uad$ has a box form (cf. \cite{troltzsch, CRT2008, CT2012}), one can restrict to impose the positive definiteness only in some particular directions.
	Hence, we need to introduce some  notions which are summarized in the following

	\begin{definition}
		Given $\lambdazb \in \Uad$, we call
		\[
			A_0(\lambdazb) := \left\{ x \in \Omega \setminus D: \, \abs*{ \int_0^T \pb(x,t)(\phib(x,t) - f) \, \de t - \frac{\beta}{\lambdazb(x)^3}} > 0 \right\}
		\]
		the \emph{set of strongly active constraints}, where $\phib$ and $\pb$ are respectively the forward and adjoint states associated to $\lambdazb$. 
		Then, we define the \emph{cone of critical directions} as the set 
		\[
			\mathcal{C}_0(\lambdazb) := \{ h_0 \in \mathcal{U}_R \mid \hbox{$h_0$ satisfies \eqref{eq:criticalcone}} \},
		\]
		where \eqref{eq:criticalcone} is given by
		\begin{equation}
			\label{eq:criticalcone}
			h_0(x) \, 
			\begin{cases}
				\ge 0, \quad \hbox{if $x \notin A_0(\lambdazb)$ and $\lambdazb(x) = \lambda_{\text{min}}(x)$,} \\
				\le 0, \quad \hbox{if $x \notin A_0(\lambdazb)$ and $\lambdazb(x) = \lambda_{\text{max}}(x)$,} \\
				= 0, \quad \hbox{if $x \in A_0(\lambdazb)$,}
			\end{cases}
		\end{equation}
		for almost any $x \in \Omega \setminus D$.
	\end{definition}

	Another thing to keep in mind when dealing with the second-order sufficient conditions is the so-called \emph{two-norm discrepancy}. This arises because $\Jred$ is of class $\mathcal{C}^2$ in $L^\infty(\Omega \setminus D)$, but its second derivative should be positive definite with respect to the $L^2(\Omega \setminus D)$-norm.
	However, if $\beta > 0$, this issue can be overcome and the whole analysis can be extended to $\Lxod 2$. 
	For further details on the matter, we refer the reader to \cite{CT2012, CT2015}. 
	Here we just state and prove the main result of this section.

	\begin{theorem}
		\label{thm:optcond_second}
		Let \ref{ass:omega}--\ref{ass:f}, \ref{ass:initial2} and \ref{ass:c1}--\ref{ass:c3} hold and let $\beta > 0$.
		Suppose that $\lambdazb \in \Uad$ is an optimal control satisfying  \eqref{eq:varineq}, and assume that
		\begin{equation}
			\label{eq:posdef}
			\D^2 \Jred(\lambdazb)[h_0, h_0] > 0 \quad \hbox{for any $h_0 \in \mathcal{C}_0(\lambdazb) \setminus \{0\}$.}
		\end{equation} 
		Then, there exist constants $\eps, \delta > 0$ such that 
		\begin{equation}
			\label{eq:localminim}
			\Jred(\lambda_0) \ge \Jred(\lambdazb) + \delta \norm{\lambda_0 - \lambdazb}^2_{L^2(\Omega \setminus D)}
			\quad 
			\hbox{$\forall \lambda_0 \in \Uad$ s.t. $\norm{\lambda_0 - \lambdazb}_{L^2(\Omega \setminus D)} < \eps$.}
		\end{equation}
		Consequently, $\lambdazb$ is a strict local minimiser of $\Jred$ in $\Uad$.
	\end{theorem}

	\begin{proof}
		We argue by contradiction in the same spirit of \cite[Theorem 4.1]{CRT2008}.
		Suppose that $\lambdazb \in \Uad$ is an optimal control satisfying \eqref{eq:varineq} and \eqref{eq:posdef}, but not \eqref{eq:localminim}. 
		Then, there must exist a sequence of controls $\{\lambda_0^n \}_{n \in \N} \subset \Uad$ such that $\lambda_0^n \neq \lambdazb$ for any $n \in \N$, $\lambda_0^n \to \lambdazb$ strongly in $\Lxod 2$ and 
		\begin{equation}
			\label{eq:oc2:contradiction}
			\Jred(\lambda_0^n) < \Jred(\lambdazb) + \frac{1}{n} \norm{\lambda_0^n - \lambdazb}^2_{\Lxod 2} 
			\quad \hbox{for any $n \in \N$.}
		\end{equation}
		Let us set, for any $n \in \N$, 
		\begin{equation}
			\label{eq:oc2:h0n}
			h_0^n := \frac{\lambda_0^n - \lambdazb}{\norm{\lambda_0^n - \lambdazb}_{\Lxod 2}} \in \Lxod\infty.
		\end{equation}  
		Since $\norm{h_0^n}_{\Lxod 2} = 1$ for any $n \in \N$, we can extract a (not relabelled) subsequence such that 
		\[
			h_0^n \weak \hb_0 \quad \hbox{weakly in $\Lxod 2$, for some $\hb_0 \in \Lxod 2$.}
		\]
		
		\noindent
		\textsc{Step 1.} We first want to show that $\D \Jred(\lambdazb)[\hb_0] = 0$. 
		For any $n \in \N$, we call $\phi_n = \Scal_1(\lambda_0^n)$ and $p_n = \Tcal_1(\lambda_0^n)$ the forward and adjoint states corresponding to $\lambda_0^n$, and we set $\phib := \Scal_1(\lambdazb)$ and $\pb := \Tcal_1(\lambdazb)$. 
		Moreover, we observe that by Theorem \ref{thm:contdep} and Proposition \ref{prop:lipcont_costate} the operators $\Scal$ and $\Tcal$ are both continuous from $\Lxod 2$ to $\YY$. 
		Thus, if $\lambda_0^n \to \lambdazb$ strongly in $\Lxod 2$, then we can immediately conclude that 
		\[
			\phi_n \to \phib \quad \text{and} \quad p_n \to \pb \quad \hbox{strongly in } \HT 1 H \cap \LT \infty W \cap \LT 2 {\Hx4}. 
		\]
		Then, by looking at \eqref{eq:Jprime}, since $h_0^n \weak \hb_0$ weakly in $\Lxod 2$, it is a standard matter to show that 
		\begin{align*}
			& \int_0^T \int_{\Omega \setminus D} p_n (f - \phi_n) h_0^n \, \de x \, \de t \to \int_0^T \int_{\Omega \setminus D} \pb (f - \phib) \hb_0 \, \de x \, \de t, \\
			& \int_{\Omega \setminus D} \frac{\beta}{(\lambda_0^n)^3} h_0^n \, \de x \to \int_{\Omega \setminus D} \frac{\beta}{\lambdazb^3} \hb_0 \, \de x. 
		\end{align*}
		In particular, observe that in the second limit we used that $1/(\lambda_0^n)^3 \to 1/\lambdazb^3$ strongly in $\Lxod 2$, which follows from the strong convergence of $\lambda_0^n \to \lambdazb$ in $\Lxod 2$ and the fact that $\lambda_0^n \in \Uad \subset \Lxod\infty$ uniformly.
		The limits above imply that 
		\begin{equation}
			\label{eq:oc2:limit1}
			\lim_{n \to + \infty} \D \Jcal(\lambda_0^n)[h_0^n] = \D \Jcal(\lambdazb)[\hb_0].
		\end{equation}
		Next, by Taylor expansion, we can write 
		\[
			\Jred(\lambda_0^n) = \Jred(\lambdazb) + \int_0^1 \D \Jred(z \lambda_0^n + (1 - z) \lambdazb)[\lambda_0^n - \lambdazb] \, \de z,
		\]
		which, by comparison with our assumption \eqref{eq:oc2:contradiction}, gives
		\[
			\int_0^1 \D \Jred(z \lambda_0^n + (1 - z) \lambdazb)[\lambda_0^n - \lambdazb] \, \de z < \frac{1}{n} \norm{\lambda_0^n - \lambdazb}^2_{\Lxod 2} \to 0, \quad \hbox{as $n \to +\infty$.}
		\]
		Thus, by combining this fact with \eqref{eq:oc2:h0n}--\eqref{eq:oc2:limit1} and the continuity of  $\D \Jred$, we infer that $\D \Jred(\lambdazb)[\hb] \le 0$. 
		On the other hand, since $\lambda_0^n \in \Uad$ and $h_0^n$ is defined by \eqref{eq:oc2:h0n}, it follows from the optimality condition \eqref{eq:varineq} that $\D \Jred(\lambdazb)[h_0^n] \ge 0$. 
		Then, by continuity, we deduce that $\D \Jred(\lambdazb)[\hb] \ge 0$. Hence, we get
		\begin{equation}
			\label{eq:oc2:step1}
			\D \Jred(\lambdazb)[\hb_0] = 0.
		\end{equation}

		\noindent
		\textsc{Step 2.} Here we show that $\hb_0 \in \mathcal{C}_0(\lambdazb)$. 
		To do this, we have to verify that $\hb_0$ satisfies the sign conditions \eqref{eq:criticalcone}.
		First of all, we observe that the set of functions in $\Lxod 2$ that are non-negative if $\lambdazb(x) = \lambda_{\text{min}}(x)$ and non-positive if $\lambdazb(x) = \lambda_{\text{max}}(x)$, almost everywhere, is closed and convex. 
		Thus, it is also weakly closed in $\Lxod 2$.
		Moreover, it is clear that, since $\lambda_0^n \in \Uad$, $\lambda_0^n - \lambdazb$ belongs to this set for any $n \in \N$.
		Then, by its definition \eqref{eq:oc2:h0n}, $h_0^n$ belongs to the same set for any $n \in \N$ and $\hb_0$ does as well, due to the weak closedness.
		At this point, by \textsc{Step 1}, we know that 
		\[
			\D \Jred(\lambdazb)[\hb_0] = \int_{\Omega \setminus D} \left( \int_0^T \pb(x,t) (\phib(x,t) - f) \, \de t - \frac{\beta}{\lambdazb(x)^3} \right) \hb_0(x) \, \de x = 0, 
		\]
		which immediately implies that $\hb_0(x) = 0$ if $x \in A_0(\lambdazb)$. 
		Consequently, $\hb_0$ satisfies the sign conditions \eqref{eq:criticalcone}, so it belongs to $\mathcal{C}_0(\lambdazb)$.

		\noindent
		\textsc{Step 3.} Our next aim is to show that $\hb_0 \equiv 0$. 
		To do this, it is enough to prove that 
		\begin{equation}
			\label{eq:oc2:step3}
			\D^2 \Jred(\lambdazb)[\hb_0, \hb_0] \le 0. 
		\end{equation}
		Indeed, if \eqref{eq:oc2:step3} holds, since $\hb_0 \in \mathcal{C}_0(\lambdazb)$ by \textsc{Step 2} and $\lambdazb$ satisfies \eqref{eq:posdef}, the only possibility is that $\hb_0 \equiv 0$. 
		Hence, we now show \eqref{eq:oc2:step3}.
		First, for any $n \in \N$ we set 
        \begin{equation*}
        \xi_n := \D \Scal(\lambdazb)[h_0^n], \quad P_n := \D \Tcal(\lambdazb)[h_0^n], \quad \bar{\xi} := \D \Scal(\lambdazb)[\hb_0], \quad \bar{P} := \D \Tcal(\lambdazb)[\hb_0].
        \end{equation*}
		Next, by the weak convergence $h_0^n \weak \hb_0$ and the fact that, for fixed $\lambdazb$, $\D\Scal(\lambdazb)$ and $\D\Tcal(\lambdazb)$ are linear bounded operators, we infer that 
		\[
			\xi_n \to \bar{\xi} \quad \text{and} \quad P_n \to \bar{P} \quad \hbox{weakly in $\HT 1 H \cap \LT 2 {\Hx4}$}.
		\]
		Then, by well-known compact embeddings (cf. \cite[Section 8, Corollary 4]{S1986}), it follows that, up to a subsequence,  
		\[
			\xi_n \to \bar{\xi} \quad \text{and} \quad P_n \to \bar{P} \quad \hbox{strongly in $C^0(\bar{Q_T})$.}
		\]
		Thus, by exploiting the weak and strong convergences highlighted above as well as the explicit expression \eqref{eq:Jsecond} of the second derivative of $\Jred$, it is easy to see that 
		\begin{align*}
			& \int_0^T \int_{\Omega \setminus D} P_n (f - \phib) h_0^n \, \de x \, \de t \to \int_0^T \int_{\Omega \setminus D} \bar{P} (f - \phib) \hb_0 \, \de x \, \de t, \\
			& \int_0^T \int_{\Omega \setminus D} \pb \xi_n h_0^n \, \de x \, \de t \to \int_0^T \int_{\Omega \setminus D} \bar{p} \bar{\xi} \, \hb_0 \, \de x \, \de t, \\
			& \int_{\Omega \setminus D} \frac{3 \beta}{\lambdazb^4} h_0^n \, \de x \to \int_{\Omega \setminus D} \frac{3 \beta}{\lambdazb^4} \hb_0 \, \de x. 
		\end{align*}
		Consequently, we conclude that 
		\begin{equation}
			\label{eq:oc2:limit2}
			\lim_{n \to + \infty} \D^2 \Jred(\lambdazb)[h_0^n, h_0^n] = \D^2 \Jred(\lambdazb)[\hb_0, \hb_0].
		\end{equation}
		We now use a second-order Taylor expansion, which, on account of the linearity of the derivatives and recalling that, by \eqref{eq:oc2:h0n}, $\lambda_0^n - \lambdazb = \norm{\lambda_0^n - \lambdazb}_{\Lxod 2} h_0^n$, yields
		\begin{align*}
			\Jred(\lambda_0^n) & = \Jred(\lambdazb) + \D \Jred(\lambdazb)(\lambda_0^n - \lambdazb) \\
			& \quad + \mezzo \norm{\lambda_0^n - \lambdazb}^2_{\Lxod 2} \int_0^1 \D^2 \Jred(z \lambda_0^n + (1 - z) \lambdazb)[h_0^n, h_0^n] \, \de z.
		\end{align*}
		Therefore, since, by the optimality condition \eqref{eq:varineq}, $\D \Jred(\lambdazb)(\lambda_0^n - \lambdazb) \ge 0$, by comparing the above expression with  \eqref{eq:oc2:contradiction}, we deduce that
		\begin{equation}
			\label{eq:oc2:bound2}
			\int_0^1 \D^2 \Jred(z \lambda_0^n + (1 - z) \lambdazb)[h_0^n, h_0^n] \, \de z < \frac{2}{n}.
		\end{equation}
		Moreover, the strong convergence $\lambda_0^n \to \lambdazb$ in $\Lxod 2$ combined with the fact that $\Jred$ is of class $\mathcal{C}^2$ (cf. Theorems \ref{thm:frechet} and \ref{thm:adjfrechet}), entails that
		\[
			\abs*{\int_0^1  \D^2 \Jred(z \lambda_0^n + (1 - z) \lambdazb)[h_0^n, h_0^n] \, \de z - \D^2 \Jred(\lambdazb)[h_0^n, h_0^n]} \to 0, \quad \text{as } n \to + \infty.
		\]
		This limit and \eqref{eq:oc2:limit2}--\eqref{eq:oc2:bound2} imply that
		\begin{align*}
			\D^2 \Jred(\lambdazb)[\hb_0, \hb_0] 
			& = \lim_{n \to + \infty} \D^2 \Jred(\lambdazb)[h_0^n, h_0^n] \\
			& \le \limsup_{n \to + \infty} \int_0^1 \D^2 \Jred(z \lambda_0^n + (1 - z) \lambdab)[h_0^n, h_0^n] \, \de z \\
			& \quad + \lim_{n \to +\infty} \left( \D^2 \Jred(\lambdazb)[h_0^n, h_0^n] - \int_0^1 \D^2 \Jred(z \lambda_0^n + (1 - z) \lambdab)[h_0^n, h_0^n] \, \de z \right) \\
			& \le 0.
		\end{align*}
		Thus \eqref{eq:oc2:step3} holds so that $\hb_0 \equiv 0$.
		
		\noindent
		\textsc{Step 4.}
		By \textsc{Step 3} we now know that $h_0^n \weak \hb_0 \equiv 0$ weakly in $\Lxod 2$. 
		Then, by the linearity of $\D \Scal(\lambdazb)$ and $\D \Tcal(\lambdazb)$, we deduce that $\bar{\xi} = \D \Scal(\lambdazb)[\hb_0] = 0$ and $\bar{P} = \D \Tcal(\lambdazb)[\hb_0] = 0$. 
		Also, by the strong convergences highlighted above, we have that $\xi_n \to 0$ and $P_n \to 0$ strongly in $C^0(\bar{Q_T})$. 
		Finally, by using the facts that $\beta > 0$, $\lambdazb \in \Uad$ and $\norm{h_0^n}_{\Lxod 2} = 1$ for any $n \in \N$, together with \eqref{eq:Jsecond} and \eqref{eq:oc2:step3}, we conclude that
		\begin{align*}
			0 & < \frac{3 \beta}{\norm{\lambdazb^4}_{\Lxod\infty}} 
			= \frac{3 \beta}{\norm{\lambdazb^4}_{\Lxod\infty}} \lim_{n \to +\infty} \norm{h_0^n}^2_{\Lx2} \\
			& \le \lim_{n \to + \infty} \int_{\Omega \setminus D} \frac{3 \beta}{\lambdazb^4} (h_0^n)^2 \, \de x \\
			& = \lim_{n \to + \infty} \D^2 \Jred(\lambdazb)[h_0^n, h_0^n] 
			- \lim_{n \to + \infty} \int_0^T \int_{\Omega \setminus D} P_n \, (f - \phib) h_0^n \, \de x \, \de t \\
			& \quad - \lim_{n \to + \infty} \int_0^T \int_{\Omega \setminus D} \pb\,  \xi_n \, h_0^n \, \de x \, \de t \\
			& = \D^2 \Jred(\lambdazb)[\hb_0, \hb_0] \le 0, 
		\end{align*}
		which leads to a contradiction. 
		The proof is now complete.
	\end{proof}
\section{Conclusions} 
In this paper, we introduced and studied an optimal control approach to the image inpainting model proposed by Bertozzi, Esedoglu, and Gillette. Our primary focus was on determining a suitable fidelity coefficient $\lambda$ to achieve faithful reconstruction of a damaged image $f$. By moving from the original polynomial double-well potential to a physically relevant logarithmic-type potential, we ensured that solutions to our evolutionary system remained strictly within the interval $[-1,1]$, thus preserving the range of the inpainted image at all times.
After establishing the well-posedness of the state system, we formulated an optimal control problem whose goal was twofold: maintain fidelity to the undamaged image $f$ in $\Omega\backslash D$ and promote sufficiently large values of $\lambda$ where fine details of the image need to be restored. We proved the existence of an optimal control $\lambda$ under the proposed constraints and then derived first-order necessary optimality conditions using an adjoint system. Subsequently, we addressed second-order sufficient conditions, which provide deeper insight into the local optimality of solutions.

Overall, our results clarify the role of the fidelity parameter in the context of Cahn--Hilliard-type inpainting and highlight how optimal control strategies can be designed to tune the reconstruction process. These theoretical findings can be of direct use in numerical algorithms, where spatially varying fidelity weights are beneficial in capturing local image features and enhancing the visual quality of the restored regions.

Possible directions for future work include:
to extend the analysis to grayscale or colour images,
to develop, based on our approach, a reconstruction procedure, to incorporate additional regularization or anisotropic effects in the fidelity term to better handle images with sharp edges or specific textures and, finally, to investigate dynamic control strategies in both space and time, potentially allowing the fidelity parameter to adapt continuously as the inpainting evolves.
We believe that these extensions, alongside the theoretical framework established here, will further promote the applicability and robustness of PDE-based image inpainting models in real-world scenarios.

\appendix

\section{Appendix}


    To provide further motivation on why it is important to choose the values of $\lambda$ large enough, we complement the results obtained in \cite{BEG2007} with an additional one describing the behaviour of the solution to \eqref{eq:phi0}--\eqref{eq:ic0} in the undamaged domain $\Omega \setminus D$.
    We recall that, if $f \in \Cx 2$ and $F$ is a regular polynomial double-well potential, in \cite{BEG2007} it was shown that, if $\lambda_0 \to +\infty$, the stationary solutions to \eqref{eq:phi0}--\eqref{eq:ic0} converge in $D$ to the solutions of the following boundary value problem:
    \begin{alignat*}{2}
		& - \Delta \left( - \eps \Delta \phi + \frac{1}{\eps} F'(\phi) \right) = 0 \qquad && \hbox{in $D$}, \\
		& \phi = f && \hbox{on $\partial D$}, \\
		& \partial_{\n} \phi = \partial_{\n} f && \hbox{on $\partial D$}. 
	\end{alignat*}
    At the same time, the stationary solution exactly matches $f$ in the undamaged domain $\Omega \setminus D$.
    This means that, as $\lambda_0$ grows larger, one can expect a better inpainting result.
    In particular, we stress that here $f$ is supposed to be regular: for instance, it can be seen as an approximation of the original image $f \in \Lx \infty$ with a regularising convolution kernel.

    With the following Proposition \ref{prop:asympt} we show that also the solution to the evolutionary inpainting model \eqref{eq:phi0}--\eqref{eq:ic0} stays close to the original image in $\Omega \setminus D$ as time grows larger, if the fidelity coefficient $\lambda_0$ is chosen large enough.
    Also in this case it is necessary to regularise the given image $f$ in some way.
    In particular, we assume to be in possession of an inpainted image $\phib$ in $D$ and, looking at our equation, we regularise $f$ as a solution $\tilde{f}$ to the following elliptic boundary value problem:
    \begin{alignat}{2}
		& - \Delta \left( - \eps \Delta \ftil + \frac{1}{\eps} F'(\ftil) \right) = 0 \qquad && \hbox{in $\Omega \setminus D$}, \label{eq:f} \\
		& \partial_{\n} \ftil = \partial_{\n} \left( - \eps \Delta \ftil + \frac{1}{\eps} F'(\ftil) \right) = 0 \quad && \hbox{on $\partial \Omega$}, \label{eq:bcf1} \\
		& \ftil = \phib, \quad \partial_{\n} \ftil = \partial_{\n} \phib && \hbox{on $\partial D$}. \label{eq:bcf2} 
	\end{alignat}
    We argue that this is a good phase-field approximation of the original image $f$ as it can be seen as a solution to the Euler--Lagrange equation associated with the Ginzburg--Landau energy.
    Hence, we show the following result.

    \begin{proposition}
        \label{prop:asympt}
        Let \ref{ass:omega}--\ref{ass:F} and \ref{ass:iniz} hold and let $\lambda_0 >0$ be a constant.
        Assume that $(\phi, \mu)$ is a weak solution to the following evolutionary problem in $\Omega \setminus D$:
        \begin{alignat}{2}
		& \partial_t \phitil - \Delta \mutil = \lambda_0 (\ftil - \phitil) \qquad && \hbox{in $(\Omega \setminus D) \times (0,+\infty)$}, \label{eq:phi:asymp} \\
		& \mutil = - \eps \Delta \phitil + \frac{1}{\eps} F'(\phitil) && \hbox{in $(\Omega \setminus D) \times (0,+\infty)$}, \label{eq:mu:asymp} \\
		& \partial_{\n} \phitil = \partial_{\n} \mutil = 0 && \hbox{on $\partial \Omega \times (0,+\infty)$}, \label{eq:bc:asymp1} \\
        & \phitil = \phib, \quad \partial_{\n} \phitil = \partial_{\n} \phib && \hbox{on $\partial D \times (0,+\infty)$}. \label{eq:bc:asymp2} \\
		& \phitil(0) = \phitil_0 && \hbox{in $\Omega \setminus D$}, \label{eq:ic:asymp}
	    \end{alignat}
        with regularities 
        \begin{align*}
            & \phitil \in \HT 1 {(H^1(\Omega \setminus D))^*} \cap \LT 2 {H^2(\Omega \setminus D)}, \\
            & \mutil \in \LT 2 {H^1(\Omega \setminus D)},
        \end{align*}
        for any positive time $T > 0$.
        Finally, let $\ftil \in \Hxod2$ be a weak solution to \eqref{eq:f}--\eqref{eq:bcf2} in the variational sense.

        Then, if $\lambda_0$ is chosen such that
        \[
            \lambda_0 > \frac{K}{\eps^3},
        \]
        where the constant $K > 0$ depends only on the parameters of the system, we have that 
        \[
            \norm{\phitil(t) - \ftil}_{(H^1(\Omega \setminus D))^*} \to 0 \quad \hbox{exponentially as $t \to + \infty$.}
        \]
    \end{proposition}

    \begin{proof}
        We start by subtracting equations \eqref{eq:f}--\eqref{eq:bcf2} for the approximated image $\ftil$ to equations \eqref{eq:phi:asymp}--\eqref{eq:ic:asymp}. 
        In order to use a more compact notation, we introduce the new variables 
        \[
            \psi := \phitil - \ftil, \quad \zeta := \mutil - \left( - \eps \Delta \ftil + \frac{1}{\eps} F'(\ftil) \right). 
        \]
        Then, by recalling that $\ftil$ does not depend on time, we see that they satisfy the following system:
        \begin{alignat}{2}
		  & \partial_t \psi - \Delta \zeta + \lambda_0 \psi = 0 \qquad && \hbox{in $(\Omega \setminus D) \times (0,+\infty)$}, \label{eq:psi:asymp} \\
		  & \zeta = - \eps \Delta \psi + \frac{1}{\eps} (F'(\phitil) - F'(\ftil)) \quad && \hbox{in $(\Omega \setminus D) \times (0,+\infty)$}, \label{eq:zeta:asymp} \\
		  & \partial_{\n} \psi = \partial_{\n} \zeta = 0 && \hbox{on $\partial \Omega \times (0,+\infty)$}, \label{eq:bc:asymp12} \\
            & \psi = 0, \quad \partial_{\n} \psi = 0 && \hbox{on $\partial D \times (0,+\infty)$}. \label{eq:bc:asymp22} \\
		  & \psi(0) = \phitil_0 && \hbox{in $\Omega \setminus D$}. \label{eq:ic:asymp2}
	    \end{alignat}
        Now we observe that $\psi$ satisfies homogeneous Neumann boundary conditions on all of $\partial (\Omega \setminus D)$ for any positive time $t \in (0,+\infty)$. Then, the inverse Neumann--Laplacian operator $\mathcal{N}^{-1} \psi(t)$ is well-defined for any $t \in (0,+\infty)$. 
        Thus, we can multiply equation \eqref{eq:psi:asymp} by $\mathcal{N}^{-1} \psi$ and integrate over $\Omega \setminus D$ to obtain that
        \begin{equation}
            \label{eq:asympt:test}
            \mezzo \ddt \norm{\psi}^2_{(H^1(\Omega \setminus D))^*} + \lambda_0 \norm{\psi}^2_{(H^1(\Omega \setminus D))^*} + \int_{\Omega \setminus D} \zeta \psi \, \de x = 0.
        \end{equation}
        Now, by using equation \eqref{eq:zeta:asymp}, the boundary conditions \eqref{eq:bc:asymp12}--\eqref{eq:bc:asymp22} and hypothesis \ref{ass:F} on $F$, we infer that
        \begin{align*}
            & \int_{\Omega \setminus D} \zeta \psi \, \de x = 
            \eps \int_{\Omega \setminus D} \abs{\nabla \psi}^2 \, \de x + \frac{1}{\eps} \int_{\Omega \setminus D} (F'(\phitil) - F'(\ftil)) (\phitil - \ftil) \, \de x \\
            & \quad = \eps \int_{\Omega \setminus D} \abs{\nabla \psi}^2 \, \de x + \frac{1}{\eps} \int_{\Omega \setminus D} \underbrace{(F_0'(\phitil) - F_0'(\ftil)) (\phitil - \ftil)}_{\ge \, 0} \, \de x + \frac{1}{\eps} \int_{\Omega \setminus D} (F_1'(\phitil) - F_1'(\ftil)) (\phitil - \ftil) \, \de x \\
            & \quad \ge \eps \int_{\Omega \setminus D} \abs{\nabla \psi}^2 \, \de x - \frac{\theta}{\eps} \int_{\Omega \setminus D} \abs{\psi}^2 \, \de x,
        \end{align*}
        where $\theta > 0$ is given by \ref{ass:F}. 
        Next, we observe that, by interpolation inequalities and Young's inequality with some $\sigma > 0$,
        \begin{align*}
            \frac{\theta}{\eps} \int_{\Omega \setminus D} \abs{\psi}^2 \, \de x 
            & \le \frac{\theta}{\eps} C_{\Omega \setminus D} \norm{\nabla \psi}^2_{L^2(\Omega \setminus D)} \norm{\psi}^2_{(H^1(\Omega \setminus D))^*} \\
            & \le \frac{\theta C_{\Omega \setminus D}}{2 \eps} \sigma \norm{\nabla \psi}^2_{L^2(\Omega \setminus D)} + \frac{\theta C_{\Omega \setminus D}}{2 \eps \sigma} \norm{\psi}^2_{(H^1(\Omega \setminus D))^*},
        \end{align*}
        where the constant $C_{\Omega \setminus D} > 0$ depends only on the set $\Omega \setminus D \subset \R^2$. 
        Consequently, from \eqref{eq:asympt:test} we obtain that 
        \[
            \mezzo \ddt \norm{\psi}^2_{(H^1(\Omega \setminus D))^*} + \left( \lambda_0 - \frac{\theta C_{\Omega \setminus D}}{2 \eps \sigma} \right) \norm{\psi}^2_{(H^1(\Omega \setminus D))^*} + \left( \eps - \frac{\theta C_{\Omega \setminus D}}{2 \eps} \sigma \right) \norm{\nabla \psi}^2_{L^2(\Omega \setminus D)} \le 0.
        \]
        Then, if we choose $\sigma > 0$ such that 
        \[
            \sigma = \frac{\eps^2}{\theta C_{\Omega \setminus D}},
        \]
        we conclude that 
        \begin{equation}
        \label{eq:asympt:estimate}
            \mezzo \ddt \norm{\psi}^2_{(H^1(\Omega \setminus D))^*} + \left( \lambda_0 - \frac{\theta^2 C^2_{\Omega \setminus D}}{2 \eps^3} \right) \norm{\psi}^2_{(H^1(\Omega \setminus D))^*} + \frac{\eps}{2} \norm{\nabla \psi}^2_{L^2(\Omega \setminus D)} \le 0,
        \end{equation}
        for any $t \in (0, + \infty)$.
        From \eqref{eq:asympt:estimate}, by application of Gronwall's inequality, we immediately deduce that, if 
        \[
            \lambda_0 > \frac{\theta^2 C^2_{\Omega \setminus D}}{2 \eps^3},
        \]
        then 
        \[
            \norm{\phitil(t) - \ftil}^2_{(H^1(\Omega \setminus D))^*} \to 0 \quad \hbox{exponentially as $t \to + \infty$,}
        \]
        by recalling the definition of $\psi$ as $\phitil - \ftil$.
    \end{proof}

    \begin{remark}
        We mention that Proposition \ref{prop:asympt} is a \emph{formal} result, in the sense that we assume the existence of a solution to \eqref{eq:phi:asymp}--\eqref{eq:ic:asymp} without actually proving it. 
        However, our aim was to show that, under suitable assumptions, large values of $\lambda_0$ guarantee that the inpainted image stays close to a suitable phase-field approximation of the original image for all times. 
        This gives some further motivation on why it is crucial to find optimal large values of $\lambda_0$, for instance through our proposed optimal control procedure.
    \end{remark}

\noindent
{\bf Acknowledgments.} 
The authors wish to thank the anonymous reviewers, who carefully read the manuscript and provided many comments that improved the quality of the paper.
C.~Cavaterra, M.~Fornoni, and M.~Grasselli are members of GNAMPA (Gruppo Nazionale per l'Ana\-li\-si Matematica, la Probabilit\`{a} e le loro Applicazioni) of INdAM (Istituto Nazionale di Alta Matematica). 
C.~Cavaterra, M.~Fornoni and M.~Grasselli have been partially supported by the MIUR-PRIN Grant 2020F3NCPX ``Mathematics for Industry 4.0 (Math4I4)''.
C.~Cavaterra has been partially supported by the MIUR-PRIN Grant 2022  
	``Partial differential equations and related geometric-functional inequalities''.
    Elena Beretta's research has been partially supported by NYUAD Science Program Project Fund AD364.
This research is part of the activities of ``Dipartimento di Eccellenza 2023-2027'' of Universit\`a degli Studi di Milano (C.~Cavaterra) and Politecnico di Milano (M.~Grasselli).

	
	\footnotesize

\end{document}